\documentclass[11pt,reqno]{amsart}
\usepackage[utf8]{inputenc}

\usepackage{amssymb, amsmath, amsthm}
\usepackage{hyperref}
\usepackage[alphabetic,lite]{amsrefs}
\usepackage{verbatim}
\usepackage{amscd}   
\usepackage[all]{xy} 
\usepackage{youngtab} 
\usepackage{young} 
\usepackage{ytableau}
\usepackage{tikz}
\usepackage{ mathrsfs }
\usepackage{cases}
\usepackage{array}
\usepackage{cellspace}
\usepackage{calligra,mathrsfs}
\usepackage{bm}
\usepackage{graphicx}
\usepackage{rank-2-roots}
\usepackage{float}

\newcommand{\defi}[1]{{\bf\upshape\sffamily #1}}
\DeclareMathOperator{\ShHom}{\mathscr{H}\text{\kern -3pt {\calligra\large om}}\,}

\renewcommand{\a}{\alpha}
\renewcommand{\b}{\beta}
\newcommand{\bw}{\bigwedge}

\def\HH{{\mathbf H}}

\def\kk{{\mathbf k}}
\renewcommand{\ll}{\lambda}
\def\LL{{\mathbf L}}

\newcommand{\onto}{\twoheadrightarrow}
\newcommand{\oo}{\otimes}
\newcommand{\pd}{\partial}

\newcommand{\Vek}{\mathcal{V}_{\kk}}

\newcommand{\bbs}{\mathbb{S}}

\newcommand{\End}{\operatorname{End}}
\newcommand{\Ext}{\operatorname{Ext}}

\newcommand{\GL}{\operatorname{GL}}
\newcommand{\Hom}{\operatorname{Hom}}

\newcommand{\Mk}{{\operatorname{\bf Mod}_{\kk}}}

\newcommand{\RG}{\operatorname{\bf Rep}\Gamma^d_{\kk}}
\newcommand{\SL}{\operatorname{SL}}

\newcommand{\Sym}{\operatorname{Sym}}
\newcommand{\Tor}{\operatorname{Tor}}

\newcommand{\coker}{\operatorname{coker}}

\renewcommand{\ker}{\operatorname{ker}}

\newcommand{\sgn}{\operatorname{sgn}}

\newcommand{\bb}[1]{\mathbb{#1}}

\newcommand{\mc}[1]{\mathcal{#1}}
\newcommand{\mf}[1]{\mathfrak{#1}}
\newcommand{\ol}[1]{\overline{#1}}
\newcommand{\op}[1]{\operatorname{#1}}

\newcommand{\ul}[1]{\underline{#1}}

\def\PP{{\mathbf P}}
\def\lra{\longrightarrow}

\newcommand{\xra}{\xrightarrow}

\newtheorem{theorem}{Theorem}[section]
\newtheorem*{theorem*}{Theorem}
\newtheorem*{problem*}{Problem}
\newtheorem{lemma}[theorem]{Lemma}

\newtheorem{proposition}[theorem]{Proposition}
\newtheorem{corollary}[theorem]{Corollary}
\newtheorem*{corollary*}{Corollary}

\newtheorem*{main-thm*}{Main Theorem}
\newtheorem*{stable-flag*}{Stable Cohomology Theorem on Flag Varieties}
\newtheorem*{stable-dd*}{Stable Cohomology Calculation for $\ll=(-d,d)$}
\newtheorem*{stable-recursion-hooks*}{Stable Cohomology Recursion for $\ll=(-a-b,a,1^b)$}
\newtheorem*{stable-proj*}{Stable Cohomology Theorem on Projective Space}
\newtheorem*{stable-hooks*}{Stable Cohomology Calculation for Hooks}
\newtheorem*{stable-twocolumn-hooks*}{Duality Theorem for 2-Column Partitions and Hooks}
\newtheorem*{stable-truncated-pows*}{Stable Cohomology Calculation for Truncated Powers}
\newtheorem*{stable-vanishing*}{Polynomial Functors with Vanishing Stable Cohomology}
\newtheorem*{vanishing-Koszul*}{Vanishing Theorem for Finite Length Koszul Modules}

\theoremstyle{definition}

\newtheorem*{definition*}{Definition}

\newtheorem{example}[theorem]{Example}

\theoremstyle{remark}
\newtheorem{remark}[theorem]{Remark}
\newtheorem*{remark*}{Remark}

\numberwithin{equation}{section}

\tikzset{
  treenode/.style = {align=center, inner sep=0pt, text centered,solid,thin,
    font=\sffamily},
  arn_n/.style = {treenode, circle, white, font=\sffamily\bfseries, draw=black,
    fill=black, text width=.5em},
  arn_nl/.style = {treenode, circle, white, font=\sffamily\bfseries, draw=black,
    fill=black, text width=1.5em},  
  arn_r/.style = {treenode, circle, red, draw=red, 
    text width=.5em, very thick},
  arn_v/.style = {treenode, circle, black, font=\sffamily\bfseries, draw=black, text width=1.2em},
  arn_x/.style = {treenode, rectangle, draw=black,
    minimum width=.5em, minimum height=0.5em},
  dott/.style={edge from parent/.style={dotted, very thick,circle,draw}},
  emph/.style={edge from parent/.style={dashed, very thick,circle,draw}},
  norm/.style={edge from parent/.style={solid,thin,circle,draw}}
}
\makeatletter
\def\labelbox#1{%
  \hbox{%
    \setbox\z@=\hbox{$\m@th\labelstyle{\,#1\,}$}%
    \setbox\tw@=\hbox{$\m@th\labelstyle\,$}%
    \dimen@=\ht\z@ \advance\dimen@ by \wd\tw@ \ht\z@=\dimen@
    \dimen@=\dp\z@ \advance\dimen@ by \wd\tw@ \dp\z@=\dimen@
    \box\z@
  }%
}
\makeatother
\begin{document}

\title{Stable sheaf cohomology and Koszul--Ringel duality}

\author{Claudiu Raicu}
\address{Department of Mathematics, University of Notre Dame, 255 Hurley, Notre Dame, IN 46556\newline
\indent Institute of Mathematics ``Simion Stoilow'' of the Romanian Academy}
\email{craicu@nd.edu}

\author{Keller VandeBogert}
\address{Department of Mathematics, University of Kentucky, 719 Patterson Office Tower, Lexington, KY 40506}
\email{keller.v@uky.edu}

\subjclass[2020]{Primary 20G05, 20G10, 14M15, 13D02}

\date{\today}

\keywords{Stable cohomology, Koszul duality, Ringel duality, Schur functors, Weyl functors, polynomial functors}

\begin{abstract} 
 We identify a close relationship between stable sheaf cohomology for polynomial functors applied to the cotangent bundle on projective space, and Koszul--Ringel duality on the category of strict polynomial functors as described in the work of Cha\l upnik, Krause, and Touz\'e. Combining this with recent results of Maliakas--Stergiopoulou we confirm a conjectured periodicity statement for stable cohomology. In a different direction, we find a remarkable invariance property for $\Ext$ groups between Schur functors associated to hook partitions, and compute all such extension groups over a field of arbitrary characteristic. We show that this is further equivalent to the calculation of $\Ext$ groups for partitions with $2$ rows (or $2$ columns), and as such it relates to Parker's recursive description of $\Ext$ groups for $\SL_2$-representations. Finally, we give a general sharp bound for the interval of degrees where stable cohomology of a Schur functor can be non-zero.
\end{abstract}

\maketitle

\section{Introduction}\label{sec:intro}

The study of polynomial representations of the general linear group over the complex numbers dates back to Issai Schur’s 1901 thesis. Over the past century, Schur functors have played a central role in many areas of mathematics. Notably, they are linked to sheaf cohomology groups on flag varieties, which (by the Borel--Weil--Bott theorem) admit explicit descriptions in terms of Schur functors. In modern treatments, the theory is developed through the framework of strict polynomial functors or representations of Schur algebras, and it can be formulated over an arbitrary commutative base ring $\kk$. Despite substantial progress, fundamental questions (such as determining extension groups or sheaf cohomology groups) remain largely open outside the characteristic zero field setting. The goal of our paper is to illustrate new connections between the study of polynomial functors and sheaf cohomology, and to showcase several applications in both settings. To that end,
\begin{itemize} 
 \item we identify an interesting relationship between the study of extension groups between polynomial functors, and sheaf cohomology calculations, by explaining how Koszul--Ringel duality on the category of polynomial functors can be realized as a stable sheaf cohomology functor; 
 \item we derive a periodicity statement for stable sheaf cohomology confirming \cite{GRV}*{Conjecture~4.2};
 \item we find a remarkable invariance property for extension groups between Schur (or Weyl) functors associated to hook partitions, which is valid over an arbitrary commutative ring $\kk$;
 \item when $\kk$ is a field, we use results from our earlier work \cite{RV} on stable cohomology to compute such extension groups for hook partitions and $2$-row/column partitions;
 \item we end with a general bound for the projective dimension of Weyl functors, and correspondingly for the vanishing of stable cohomology for Schur functors.
\end{itemize} 

\medskip

\noindent{\bf Koszul duality and stable cohomology.} We work throughout over a commutative ring $\kk$, and we consider the category $\RG$ of strict polynomial functors of degree~$d$ over $\kk$ (following \cites{Krause-duality,touze,chalup-Adv,fri-sus}). When $\kk$ is a field, it is shown in \cite{RV}*{Section~4} that for a polynomial functor $\mc{P}$ of degree~$d$, the sheaf cohomology groups $H^j(\PP^N,\mc{P}(\Omega))$ stabilize when $N\geq d$, where $\Omega$ denotes the cotangent bundle on~$\PP^N$. We revisit the theory in Section~\ref{sec:stcoh-revisit}, and explain how it extends to an arbitrary commutative base ring~$\kk$. We denote $H^j_{st}(\mc{P})=H^j(\PP^N,\mc{P}(\Omega))$ for $N\geq d$, and refer to it as the stable cohomology of $\mc{P}(\Omega)$. One of our goals is to connect stable cohomology to Koszul--Ringel duality on the (derived) category of polynomial functors. To that end we write $\Vek$ for the category of finitely generated projective $\kk$-modules, and define
\begin{equation}\label{eq:def-Hjst-functor}
 \HH^j_{st}:\RG\lra\RG,\quad\HH^j_{st}(\mc{P})(V) = H^j_{st}(\mc{P}(\Omega\oo_{\kk}V))\quad\text{ for }\mc{P}\in\RG,\ V\in\Vek.
\end{equation}
Recall from \cite{Krause-duality} that $\RG$ admits an internal tensor product $-\oo_{\Gamma^d_{\kk}}-$ which is right exact, and let 
\begin{equation}\label{eq:defXi-RG}
\Xi:\RG\lra\RG,\quad\Xi(\mc{P})=\bw^d\oo_{\Gamma^d_{\kk}}\mc{P},
\end{equation}
where $\bw^d$ denotes the exterior power functor. This functor formalizes the construction described by Akin--Buchsbaum \cite[Section 6]{AB2} passing from divided to exterior powers; in characteristic $0$, $\Xi$ specializes to the transpose duality functor \cite{sam-snowden}*{Section~3.3.8}, and as such it categorifies the classical $\omega$-involution \cite[Chapter 2]{macdonald}  on the ring of symmetric functions. In arbitrary characteristic the functor $\Xi$ is only right exact, and our first main result shows that the higher derived functors of $\Xi$ may be described using stable cohomology.

\begin{theorem}\label{thm:stcoh-transpose-duality}
    We have that $\HH^d_{st}=\Xi$, and more generally $\HH^j_{st} = \LL_{d-j}\Xi$ are the left derived functors of $\Xi$. If~$\kk$ is a field $\text{ and }^{\vee}$ denotes the vector space dual, then we have natural isomorphisms
    \[ \Ext_{\RG}^j\left(\mc{P},\bw^d\right) = H^{d-j}_{st}(\mc{P})^{\vee}\quad\text{ for all }j.\]
\end{theorem}

\begin{remark}\label{rem:Koszul duality}
We refer the reader to \cites{Krause-duality,touze,chalup-Adv} and Section~\ref{subsec:poly-fun} for more background, and summarize here some facts and notation. Krause proves that $\Xi$ induces a duality at the level of the derived category, and relates it to Ringel duality for the Schur algebra. At the level of the derived category $\Xi$ induces a quasi-inverse to the Koszul duality functor $\Theta$ from \cite{chalup-Adv} (where $\kk$ is a field), also referred to as Ringel duality in \cite{touze} (where $\kk$ is a PID). One has that $\Theta = {\bf R}\mathscr{H}om(\bw^d,-)$, where $\mathscr{H}om$ denotes the internal $\Hom$ functor in $\RG$ (denoted by $\bb{H}$ in \cite{touze}). In \cite{chalup-Adv}*{Definition~2.3}, the functor $\Xi$ is denoted by $\Theta_p$. The functor $\Xi$ also appears in \cite{Krause-book}*{Chapter 8} where it is denoted by $\Omega$.
\end{remark}

If we write $\bb{S}_{\ll}$ for the Schur functor associated to a partition $\ll$, then it follows from Theorem~\ref{thm:stcoh-transpose-duality} and \cite{RV}*{Theorem~6.12} that the following non-canonical isomorphisms hold when $\kk$ is a field:
\begin{equation}\label{eq:strange-duality}
 \Ext^j_{\operatorname{\bf Rep}\Gamma^{m+n}_{\kk}}\left(\bb{S}_{(2^n,1^{m-n})},\bw^{m+n}\right) \simeq \Ext^{n-j}_{\operatorname{\bf Rep}\Gamma^{m+1}_{\kk}}\left(\bb{S}_{(n+1,1^{m-n})},\bw^{m+1}\right)\quad\text{for all $m\geq n$ and all $j$}.
\end{equation}
Over a general ring $\kk$ however, there is no direct connection between the groups in \eqref{eq:strange-duality}. One can show using \cite{FRRZY}*{Theorem~1.1} (see also \cite{GRV}*{Conjecture~2.3}) that the complexes ${\bf R}\Hom\left(\bb{S}_{(2^n,1^{m-n})},\bw^{m+n}\right)$ and ${\bf R}\Hom\left(\bb{S}_{(n+1,1^{m-n})},\bw^{m+1}\right)$ are derived dual up to a homological shift, hence the groups in \eqref{eq:strange-duality} are in fact dual vector spaces (see Example~\ref{ex:W22-Ext}). In the case $\kk=\bb{Z}$ we get instead
\[ 
\Ext^j_{\operatorname{\bf Rep}\Gamma^{m+n}_{\bb{Z}}}\left(\bb{S}_{(2^n,1^{m-n})},\bw^{m+n}\right) = \Ext^{n+1-j}_{\operatorname{\bf Rep}\Gamma^{m+1}_{\bb{Z}}}\left(\bb{S}_{(n+1,1^{m-n})},\bw^{m+1}\right)^{*}\quad\text{for all $m\geq n$ and all $j$},
\]
where for a finite abelian group $M$, we let $M^* = \Hom_{\bb{Z}}(M,\bb{Q}/\bb{Z})=\Ext^1_{\bb{Z}}(M,\bb{Z})$ denote its Pontryagin dual. This is abstractly isomorphic to $M$, so we can get again a non-canonical identification as in \eqref{eq:strange-duality}, but with a different cohomological shift (see also \cite{FRRZY}*{Theorem~1.2(iv)}).

\begin{remark}\label{rem:res-wedge-versus-Schur}
 It is a consequence of Kuhn duality that $\Ext^j(\bb{S}_{\ll},\bw^d) = \Ext^j(\bw^d,\bb{W}_{\ll})$ (where here and henceforth $\Ext$ is computed in the category $\RG$, which we suppress from notation when the context is clear). In Section~\ref{subsec:spec-Ext} we describe some specialization complexes that compute $\Ext(\bw^d,\mc{P})$ for a general polynomial functor $\mc{P}$. However, it seems difficult to understand \eqref{eq:strange-duality} through these complexes. Instead, the proof of \eqref{eq:strange-duality} in \cite{RV} is based on explicit resolutions of Schur functors by tensor products of exterior powers following \cite{AB1}. We will return to discuss such resolutions below.
\end{remark}

One can use Theorem~\ref{thm:stcoh-transpose-duality} in the opposite direction, starting with established results for $\Ext$ groups and deriving consequences to stable cohomology. Building on the main results of \cite{periodicity-mal-ste}, we deduce the following periodicity result for stable cohomology \cite{GRV}*{Conjecture~4.2}.

\begin{theorem}\label{thm:periodicity}
  Consider a partition $\mu=(\mu_1,\cdots,\mu_{\ell})$ with $\mu_{\ell}>0$, and write 
 \[\mu[q] = (\mu,1^q) = (\mu_1,\cdots,\mu_{\ell},1,\cdots,1)\] 
 for the partition obtained from $\mu$ by appending $q$ parts of size $1$. If $\kk$ is a field of characteristic $p>0$ and $q=p^r>|\mu|-{\ell}$ then 
 \[H^j_{st}(\bb{S}_{\mu}) = H^{j+q}_{st}(\bb{S}_{\mu[q]}).\]
 Consequently, if we write $H^j_{st}(\ll)$ for the stable cohomology groups associated to a weight $\ll$ on the full flag variety as in \cite{RV}*{Stable Cohomology Theorem on Flag Varieties}, then
 \[H^j_{st}(-|\mu|,\mu) = H^{j+q}_{st}(-|\mu|-q,\mu[q]).\]
\end{theorem}

\medskip

\noindent{\bf Invariance of Ext groups for hook partitions.} Consider hook partitions $\ll=(A,1^B)$ and $\mu=(a,1^b)$ of the same size $d$, and assume further that we fix the difference $\Delta=A-a\geq 0$. One can start with $\ll=(d)$, $\mu=(d-\Delta,1^{\Delta})$, and slide both hooks to the left to obtain all such $\ll$, $\mu$ (below is the case $d=6$ and $\Delta=2$):

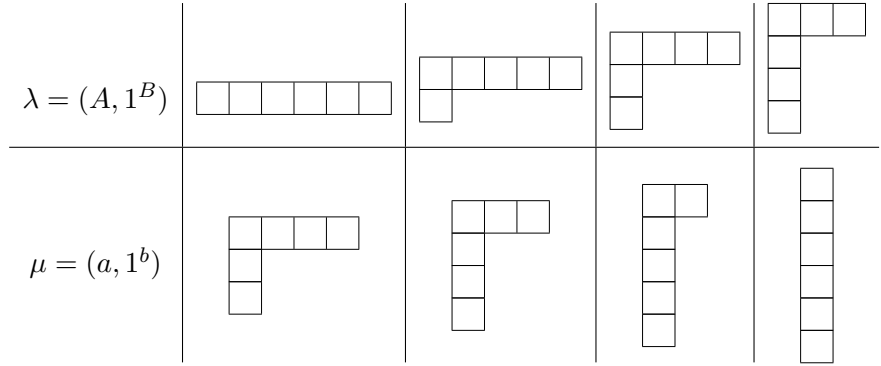
\newcommand\Tstrut{\rule{0pt}{10ex}}         
\newcommand\Bstrut{\rule[-0.9ex]{0pt}{0pt}}   
\[
\Yvcentermath1
\begin{array}{c|c|c|c|c}
 \raisebox{2.3ex}{$\ll=(A,1^B)$} & \raisebox{1.8ex}{\yng(6)} & \raisebox{1.2ex}{\yng(5,1)} & \raisebox{0.6ex}{\yng(4,1,1)} & \raisebox{0.3ex}{\yng(3,1,1,1)} \\
\hline
 \mu=(a,1^b) & \yng(4,1,1) &  \yng(3,1,1,1) &  \yng(2,1,1,1,1) & \yng(1,1,1,1,1,1)\Tstrut\\
\end{array}
\]
We prove that over an arbitrary commutative ring $\kk$, the groups $\Ext^j\left(\bb{S}_{\ll},\bb{S}_{\mu}\right)$ only depend on $d$ and $\Delta$, so they are invariant under sliding the hooks simultaneously.

\begin{theorem}\label{thm:introExtComputation}
    Over any commutative ring $\kk$ there exist isomorphisms
    $$\Ext^j \left(\bb{S}_{(A,1^B)}, \bb{S}_{(a,1^b)} \right) = \Ext^j \left(\bb{S}_{(A-1 , 1^{B+1})}, \bb{S}_{(a-1,1^{b+1})} \right).$$
    In particular, if $\kk$ is a field and if we let $d=a+b$, $\Delta=A-a$ as before, then
    $$\Ext^j \left( \bb{S}_{(A,1^B)}, \bb{S}_{(a,1^b)} \right) = H^{d-j}_{st} \left(\bbs_{(\Delta+1,1^{d-\Delta-1})} \right)^{\vee}.$$
\end{theorem}

Notice that we have an involution $(\ll,\mu)\lra(\mu',\ll')$ on the pairs of hooks with fixed $\Delta=A-a$, where $\ll'$ denotes the conjugate partition to $\ll$, and the general identification \eqref{eq:ext-weyl-schur=eqs} below shows that $\Ext$ is invariant under this involution. Our result shows that \eqref{eq:ext-weyl-schur=eqs} is a special case of a more general invariance phenomenon in the case of hooks. Extension groups between Weyl modules for hooks have been studied in a series of papers \cites{mal-resolutions,mal-ste-hookWeyl,ste-ext2}. The explicit results for particular Ext groups in these papers illustrate special instances of the invariance property in Theorem~\ref{thm:introExtComputation}. As we explain next, when $\kk$ is a field we can use our prior work \cite{RV} to give explicit formulas for all the groups in Theorem~\ref{thm:introExtComputation}.

\medskip

\noindent{\bf Some explicit formulas for $\Ext$ groups over a field.} We next focus on the case when $\kk$ is a field of characteristic $p>0$. It is a fundamental open problem to understand $\Ext$ groups between Schur (or Weyl) modules, and the existence of non-zero morphisms is already highly non-trivial, and goes back to classical work of Andersen, Carter, Donkin, Lusztig, Payne, and many others \cites{andersen-weyl,carter-lusztig,carter-payne,donkin-ratl}. Using Kuhn duality and Koszul--Ringel duality, one has the following identifications
\begin{equation}\label{eq:ext-weyl-schur=eqs}
 \Ext^j\left(\bb{S}_{\ll},\bb{S}_{\mu}\right) = \Ext^j\left(\bb{S}_{\mu'},\bb{S}_{\ll'}\right) = \Ext^j\left(\bb{W}_{\mu},\bb{W}_{\ll}\right) = \Ext^j\left(\bb{W}_{\ll'},\bb{W}_{\mu'}\right).
\end{equation}
Moreover, the groups in \eqref{eq:ext-weyl-schur=eqs} vanish identically unless $\mu\leq\ll$ in the dominance order \cite{CPSvK}. There are useful reduction results for $\Ext$ groups, such as the row/column removal theorem (see \cite{donkin-tilting}*{Section~10} for a general statement), as well as vanishing results that come from understanding the block structure for the category of polynomial representations \cite{donkin-blocks}. Nevertheless, explicit formulas for the groups \eqref{eq:ext-weyl-schur=eqs} remain elusive in general. The most detailed understanding to date occurs in the setting of exponential functors (symmetric, exterior, divided powers), as explained in \cites{chalup-Adv,touze-bar-complexes}. 

Our goal is to study a few classes of non-exponential functors, and provide explicit formulas when  $\ll,\mu$ are (simultaneously) hook partitions, or $2$-row partitions (or equivalently by \eqref{eq:ext-weyl-schur=eqs}, $2$-column partitions). Remarkably, the $\Ext$ groups are all controlled in these cases by the stable cohomology calculations for hooks from \cite{RV}*{Section~6}. As in \cite{RV}*{(1.7)} we define the power series
\begin{equation}\label{eq:def-Atu}
 \mc{A}(t,u) = \prod_{i\geq 1}\frac{1+t\cdot u^{p^i}}{1-t^2\cdot u^{p^i}}.
\end{equation}
For $k\geq 0$ we consider the $p$-adic expansion
\[ k = \sum_{j\geq 0} k_j\cdot p^j = k_0 + k_1\cdot p + k_2\cdot p^2 + \cdots, \text{ where }0\leq k_j < p\text{ and }k_j=0\text{ for }j\gg 0,\]
and we define
\begin{equation}\label{eq:def-kbar-i}
 \ol{k}(i) = (p^{i+1}-1) - (k_0+k_1\cdot p + \cdots + k_i\cdot p^i) = \sum_{j=0}^i (p-1-k_j)\cdot p^j,
\end{equation}
making the convention that $\ol{k}(-1)=0$. Finally, we consider the power series
\begin{equation}\label{eq:E-pow-ser}
 \mc{E}_k(t,u) = \sum_{\substack{i\geq 0 \\ k_i\neq p-1}} \left(u^{\ol{k}(i-1)}+t\cdot u^{\ol{k}(i)}\right) \cdot \mc{A}(t,u^{p^i}).
\end{equation}

Theorem \ref{thm:introExtComputation} combined with the results of \cite{RV} imply that the coefficients of $\mc{E}_k(t,u)$ completely determine extensions between Schur/Weyl functors associated with hooks and 2-row/column partitions over \emph{any} field $\kk$.

\begin{theorem}\label{thm:dimExt-hook-2row}
 Let $e^j(\ll,\mu)$ denote the dimension of $\Ext^j(\bb{S}_{\ll},\bb{S}_{\mu})$ as a $\kk$-vector space. If $a+b=A+B$ then
 \begin{enumerate}
  \item For $\ll=(A,B)$ and $\mu=(a,b)$, $e^j(\ll,\mu)$ equals the coefficient of $t^j\cdot u^{A-a}$ in $\mc{E}_{a-b}(t,u)$.
  \item For $\ll=(a,b)'=(2^b,1^{a-b})$ and $\mu=(A,B)'=(2^B,1^{A-B})$, $e^j(\ll,\mu)$ equals, as in (1), the coefficient of $t^j\cdot u^{A-a}$ in $\mc{E}_{a-b}(t,u)$.
  \item For $\ll=(A,1^B)$ and $\mu=(a,1^b)$, $e^j(\ll,\mu)$ equals the coefficient of $t^{A-a-j}\cdot u^{A-a}$ in $\mc{E}_{a+B-1}(t,u)$.
 \end{enumerate}
\end{theorem}

We record an example illustrating Theorem \ref{thm:dimExt-hook-2row} below.
\begin{example}\label{ex:E3-tu}
 Suppose that $p=2$ and consider the series $\mc{E}_k(t,u)$ for $k=3$. We have $k_0=k_1=1=p-1$ and $k_i=0$ for $i\geq 2$. It follows that $\ol{k}(i)=2^{i+1}-4$ for $i\geq 1$, hence
 \[ 
 \begin{aligned}
 \mc{E}_3(t,u) &= \sum_{i\geq 2} \left(u^{2^i-4}+t\cdot u^{2^{i+1}-4}\right) \cdot \mc{A}(t,u^{2^i}) = (1+tu^4)\cdot\mc{A}(t,u^4)+(u^4+tu^{12})\cdot\mc{A}(t,u^8)+\cdots \\
 & = 1+u^4\cdot(1+t) + u^8\cdot(t+t^2) + u^{12}\cdot(1+t+t^2+t^3)+\cdots
 \end{aligned}
 \]
 It follows that the first non-zero $\Ext$ groups for $\ll\neq\mu$ occur when $A-a=4$, and we have using Theorem~\ref{thm:dimExt-hook-2row}(1)
 \[ \Ext^0\left(\bb{S}_{(11,0)},\bb{S}_{(7,4)}\right) = \Ext^1\left(\bb{S}_{(11,0)},\bb{S}_{(7,4)}\right) = \kk.\]
 If instead we use Theorem~\ref{thm:dimExt-hook-2row}(3) with $(A,B)=(5,3)$ and $(a,b)=(1,7)$ then we get
 \[\Ext^3\left(\bb{S}_{(5,1^3)},\bw^8\right) = \Ext^4\left(\bb{S}_{(5,1^3)},\bw^8\right) = \kk.\]
\end{example}

The proof of Theorem~\ref{thm:dimExt-hook-2row} is not difficult given the prior results and is presented in Section~\ref{sec:explicit-Ext}. The point is that the power series \eqref{eq:E-pow-ser} are closely related to the ones in \cite{RV}*{(1.6)} that describe stable cohomology for all hooks. If $\ll$ and $\mu$ are hook partitions, then Theorem~\ref{thm:introExtComputation} reduces the calculation of $\Ext$ groups to a stable cohomology computation. If instead $\ll,\mu$ are $2$-column partitions, then the row-removal theorem allows us to reduce to the case when $\mu$ has a single column, which is then handled using \eqref{eq:strange-duality}. 

An interesting consequence arises for $2$-row partitions, where the extension groups can be computed in the category of polynomial $\GL_2$-representations. The blocks of the category of degree $d$ polynomial representations of $\GL_n$ are described in \cite{donkin-blocks}, and taking $n\geq d$ they are equivalent to the blocks of the category of polynomial functors of degree $d$. For weights $\ll$, $\mu$ in different blocks one gets vanishing of the extension groups \eqref{eq:ext-weyl-schur=eqs}, but the converse is usually false. A consequence of Theorem~\ref{thm:dimExt-hook-2row} and \cite{RV}*{(1.10)} is that the converse does hold in the category of $\GL_2$-representations.
\begin{corollary}\label{cor:Ext-vs-block}
 Suppose that $\ll=(A,B)$ and $\mu=(a,b)$ are partitions of the same size with $A\geq a$. The weights $\ll,\mu$ belong to the same block of the category of polynomial $\GL_2$-representations if and only if there exists an index $j$ such that $\Ext^j(\bb{S}_{\ll},\bb{S}_{\mu})\neq 0$.
\end{corollary}
Notice that although $\Ext^j(\bb{S}_{\ll},\bb{S}_{\mu}) = \Ext^j_{\GL_n}(\bb{S}_{\ll},\bb{S}_{\mu})$ for all $n\geq 2$, the property of $\ll=(A,B,0,\cdots)$ and $\mu=(a,b,0,\cdots)$ being in the same block of the category of $\GL_n$-representations depends on $n$. For instance in characteristic $p=2$ one has that $\ll=(7,0)$ and $\mu=(5,2)$ belong to different blocks for $\GL_2$ (hence the $\Ext$ groups vanish, see Example~\ref{ex:E3-tu}), but $(7,0,0)$ and $(5,2,0)$ belong to the same block for $\GL_3$, hence Corollary~\ref{cor:Ext-vs-block} cannot be extended beyond $\GL_2$. For a related example in the case of hooks see \cite{RV}*{Example~6.7}.

Our results provide then an alternative perspective on \cite{parker}*{Theorem~5.1} which describes recursively the $\Ext$ groups between Weyl modules for the algebraic group $\SL_2$. Indeed, if we write $\Delta(r)$ for the Weyl module of $\SL_2$ of highest weight~$r$ as in \cite{parker}, then $\bb{W}_{(a,b)}(\kk^2)=\Delta(a-b)$ and we have
\[\Ext^j(\bb{S}_{(A,B)},\bb{S}_{(a,b)}) = \Ext^j_{\SL_2}(\Delta(a-b),\Delta(A-B)).\]
Expanding \eqref{eq:def-Atu} according to the powers of $t$, we get
\[ \mc{A}(t,u) = 1+t\left(\sum_{i\geq 1}u^{p^i}\right)+t^2\left(\sum_{i\geq 1}u^{p^i}+\sum_{i>j\geq 1}u^{p^i+p^j}\right)+t^3\left(\sum_{i\geq j\geq 1}u^{p^i+p^j}+\sum_{i>j>k\geq 1}u^{p^i+p^j+p^k}\right)+\cdots\]
which combined with \eqref{eq:E-pow-ser} allows one to describe $\Ext^j_{\SL_2}(\Delta(r),\Delta(s))$ for small values of $j$, as in \cite{erdmann-ext1,cox-erdmann}.

\medskip

\noindent{\bf Short resolutions for Schur functors.} As illustrated in \cite{RV}, an effective way of computing stable cohomology for $\bb{S}_{\mu}$ is through simple, explicit resolutions of $\bb{S}_{\mu}$ by direct sums of tensor products of exterior power functors. The study of such resolutions (or of resolutions of Weyl functors by tensor products of divided powers) was initiated by Akin and Buchsbaum \cites{AB1,AB2}, but a general explicit construction was only found by Santana and Yudin \cite{san-yud}. One drawback of their construction is that the resolutions are quite long \cite{san-yud}*{Corollary~4.5}, and the next result suggests perhaps what the optimal size of such a resolution should~be (see Example~\ref{ex:res-W222} for some comparisons between our results and previously constructed resolutions).

\begin{theorem}\label{thm:short-res-Schur}
    Consider a partition $\mu$ and the associated Schur and Weyl functors $\bb{S}_{\mu}$ and $\bb{W}_{\mu}$. Write $d=|\mu|$ for the size of~$\mu$, and $\ell=\ell(\mu)$ for the number of parts of $\mu$. The following hold over any commutative ring $\kk$.
    \begin{enumerate}
     \item There exists a (projective) resolution of $\bb{W}_{\mu}$ by tensor products of divided powers, which has length at most $d-\mu_1$.
     \item There exists a resolution of $\bb{S}_{\mu}$ by tensor products of exterior powers, which has length at most $d-\ell$.
     \item $H^j_{st}\left(\bb{S}_{\mu}\right) = 0$ for $j<\ell$ (and $j>d$).
     \item $\Ext^j\left(\bb{S}_{\mu},\bw^d\right)=0$ for $j>d-\ell$ (and $j<0$).
    \end{enumerate}
\end{theorem}

\medskip

In the special case $d=\ell$ we have $\mu=(1^d)$ and $\bb{S}_{\mu}\Omega = \Omega^d$, so the only non-vanishing cohomology occurs for $H^d_{st}(\bw^d)=\kk$. For a more interesting example, let $p=2$ and $\mu=(2^r)$, so that $\bb{S}_{\mu}=\Sym^{2^r}$. We have $d=2^r$, $\ell=1$, and it is shown in \cite{RV}*{Example~6.10} that in this case we have $H^j_{st}\left(\bb{S}_{\mu}\right)\neq 0$ for every $j$ in the interval $\ell\leq j\leq d$ (see \cite{akin-ext} for the corresponding $\Ext$ group result). For yet another example, notice that the groups in \eqref{eq:strange-duality} vanish for $j<0$ or $j>n$, which is consistent with the vanishing predicted by Theorem~\ref{thm:short-res-Schur}: if either $\mu=(2^n,1^{m-n})$ or $\mu=(n+1,1^{m-n})$, then we get $|\mu|-\ell(\mu) = n$.

\begin{remark}
    One can show that there is also a \emph{lower} bound on the length of a universal resolution of $\bb{W}_{\mu}$ as in Theorem~\ref{thm:short-res-Schur}(1), which comes from the characteristic $0$ constructions of Akin and Zelevinsky \cite{akinJT,zelevinsky}, and employs the connection with the Jacobi--Trudi identity. For partitions $\mu=(\mu_1,\cdots,\mu_{\ell})$ with $\mu_{\ell}\geq \ell-1$, their construction provides a resolution of length $\binom{\ell}{2}$. 
\end{remark}

\medskip

\noindent{\bf Organization.} In Section~\ref{sec:prelim} we recall some basic facts and notation regarding the theory of polynomial functors, and introduce the specialization complexes that are fundamental in our calculations of $\Ext$ groups. In Section~\ref{sec:stcoh-revisit} we discuss the relationship between stable cohomology and Koszul duality, and explain the proofs of Theorems~\ref{thm:stcoh-transpose-duality} and~\ref{thm:periodicity}, as well as that of Theorem~\ref{thm:short-res-Schur}. In Section~\ref{sec:hook-Ext} we prove Theorem~\ref{thm:introExtComputation} establishing invariance property for $\Ext$ groups associated with hook partitions. In Section~\ref{sec:explicit-Ext} we discuss the explicit $\Ext$ calculation in the case of hooks and $2$-row/column partitions, proving Theorem~\ref{thm:dimExt-hook-2row} and Corollary~\ref{cor:Ext-vs-block}.

\section{Preliminaries}\label{sec:prelim}

\subsection{Strict polynomial functors}
\label{subsec:poly-fun}

We let $\kk$ denote a commutative ring, and recall the theory of strict polynomial functors following \cite{Krause-duality}. We let $\Vek$ denote the category of finitely generated projective $\kk$-modules, and write $\Gamma^d_{\kk}$ for the category of degree $d$ \defi{divided powers}, whose objects are the same as for $\Vek$, and the morphisms
\[ \Hom_{\Gamma^d_{\kk}}(V,W) = \Hom_{\mf{S}_d}(V^{\oo d},W^{\oo d})\]
are defined to be the $\mf{S}_d$-equivariant morphisms between $V^{\oo d}$ and $W^{\oo d}$, where the action of the symmetric group $\mf{S}_d$ is by permuting the tensor factors. We let $\op{Mod}_{\kk}$ denote the category of $\kk$-modules, and write $\RG$ for the category of \defi{$\kk$-linear representations of $\Gamma^d_{\kk}$}, which consists of $\kk$-linear functors from $\Gamma^d_{\kk}$ to $\op{Mod}_{\kk}$. We will refer to this as the category of \defi{strict polynomial functors of degree $d$}, which in the case when $\kk$ is a field is equivalent to the category described in \cite{fri-sus}*{Section~2}. For $n\geq d$, the theory of strict polynomial functors of degree $d$ is further equivalent to that of representations of the Schur algebra $S_{\kk}(n,d)=\End_{\Gamma^d_{\kk}}(\kk^n)$, or to the category of degree $d$ polynomial representations of $\GL_n$ over $\kk$ \cites{fri-sus,green}. We will occasionally consider non-homogeneous  functors, which are just direct sums $\mc{P} = \bigoplus_d\mc{P}_d$, where $\mc{P}_d$ is polynomial of degree~$d$.

Given a partition $\ll=(\ll_1,\ll_2,\cdots)$ of $d$, denoted $\ll\vdash d$, we write $\bb{S}_{\ll}$ (resp. $\bb{W}_{\ll}$) for the \defi{Schur functor} (resp. \defi{Weyl functor}) associated to $\ll$, which is a polynomial functor of degree $d$. If we write $V^{\vee}=\Hom_{\kk}(V,\kk)$ then for a polynomial functor $\mc{P}\in\RG$, its \defi{Kuhn dual} $\mc{P}^{\#}$ is defined via
\[\mc{P}^{\#}(V)=\left(\mc{P}(V^{\vee})\right)^{\vee},\]
and an important instance of this identifies $\bb{W}_{\ll}=\bb{S}_{\ll}^{\#}$. We have moreover:
\begin{itemize}
\item If $\ll=(1^d)=(1,\cdots,1)$ then $\bb{W}_{\ll}=\bb{S}_{\ll}=\bw^d$ is the exterior power functor. 
\item If $\ll=(d)$ then $\bb{S}_{\ll}=\Sym^d$ is the symmetric power functor and $\bb{W}_{\ll}=D^d$ is the divided power functor.
\end{itemize}
For a tuple $\ul{d}=(d_1,\cdots,d_n)\in\bb{Z}^n_{\geq 0}$ we will write
\[ D^{\ul{d}} = D^{d_1} \oo D^{d_2} \oo \cdots \oo D^{d_n},\quad \bw^{\ul{d}} = \bw^{d_1} \oo \bw^{d_2} \oo \cdots \oo \bw^{d_n},\quad \Sym^{\ul{d}} = \Sym^{d_1} \oo \Sym^{d_2} \oo \cdots \oo \Sym^{d_n}.\]
The functors $D^{\ul{d}}$ with $|\ul{d}|=d_1+\cdots+d_n=d$ are projective generators for $\RG$, while $\Sym^{\ul{d}}$ are injective cogenerators. Using the action of the algebraic torus $(\kk^{\times})^n$ we get a weight space decomposition
\begin{equation}\label{eq:weight-dec-P}
 \mc{P}(\kk^n)= \bigoplus_{\ul{d}\in\bb{Z}^n_{\geq 0}}\mc{P}_{\ul{d}}.
\end{equation}
Note that since $\mc{P}\in\RG$, we have $\mc{P}_{\ul{d}}=0$ if $|\ul{d}|\neq d$. An alternative description of the weight spaces comes from the natural isomorphism \cite{Krause-book}*{Lemma~8.3.18}
\begin{equation}\label{eq:canonical-wtspace=Hom}
 \mc{P}_{\ul{d}}=\Hom(D^{\ul{d}},\mc{P})
\end{equation}
Given a polynomial functor $\mc{P}\in\RG$, the \defi{shift functor $\op{sh}(\mc{P})$} is a non-homogeneous functor defined by
\[ \op{sh}(\mc{P})(V) = \mc{P}(\kk\oplus V),\]
which then admits a decomposition
\begin{equation}\label{eq:shP-decomp}
 \op{sh}(\mc{P}) = \mc{P} \oplus \mc{P}^{(1)} \oplus \cdots,\quad\text{ where }\mc{P}^{(a)}\in{\operatorname{\bf Rep}\Gamma^{d-a}_{\kk}}.
\end{equation}
At the level of weight spaces we have an identification
\begin{equation}\label{eq:weight-shift}
 \mc{P}^{(a)}_{\ul{d}} = \mc{P}_{(a,d_1,\cdots,d_n)}.
\end{equation}

\subsection{Specialization complexes and Ext groups}
\label{subsec:spec-Ext}

This section is based on results from \cite{mal-resolutions}, \cite{mal-ste-hookWeyl}*{Section~2}, presented in a way that is most useful for our work. We fix a polynomial functor $\mc{P}$ of degree $d$. For each $n\geq 1$ we fix the standard basis $\{e_1,\cdots,e_n\}$ for $\kk^n$ and consider \defi{specialization maps} $\psi_i : \kk^{n+1} \lra \kk^n$ for $i=1,\cdots,n$, defined by
\begin{equation}\label{eq:spec-maps} 
\psi_i (e_j) = \begin{cases}
e_j & \text{ if }j\leq i, \\
e_{j-1} & \text{ if }j \geq i+1.
\end{cases}
\end{equation}
By functoriality we get maps $\mc{P}(\psi_i):\mc{P}(\kk^{n+1}) \lra \mc{P}(\kk^n)$, which by abuse of notation we will continue to denote by $\psi_i$. By taking the alternating sum of the specialization maps this defines a chain complex
\begin{equation}\label{eq:Pspec-exact}
\mc{P}(\kk^{\bullet}):\qquad \cdots \lra \mc{P}(\kk^{n+1}) \lra \mc{P}(\kk^n) \lra \mc{P}(\kk^{n-1})\lra \cdots \lra \mc{P}(\kk^1)
\end{equation}
with differential $\partial_{n+1}:\mc{P}(\kk^{n+1}) \lra \mc{P}(\kk^n)$ given by
\[ \partial_{n+1} = \psi_1-\psi_2+\cdots = \sum_{i=1}^n (-1)^{i-1}\cdot \psi_i.\]
The complex \eqref{eq:Pspec-exact} is exact (see the remarks below), but it contains several meaningful subcomplexes as described next. 

We say that a weight $\ul{d}\in\bb{Z}^n_{\geq 0}$ has \defi{full support} if $d_i > 0$ for all $i=1,\cdots,n$, and write
\[\psi_i(\ul{d}) = (d_1,\cdots,d_{i-1},d_i+d_{i+1},d_{i+2},\cdots,d_n),\]
noting that if $\ul{d}$ has full support then so does $\psi_i(\ul{d})$. Using the weight space decomposition \ref{eq:weight-dec-P}, we have that $\psi_i$ sends $\mc{P}_{\ul{d}}$ to $\mc{P}_{\psi_i(\ul{d})}$, and therefore we get a subcomplex $\mc{P}(\kk^{\bullet})^{full}$ of \eqref{eq:Pspec-exact}, where
\[ \mc{P}(\kk^n)^{full} = \bigoplus_{\ul{d}\in\bb{Z}^n_{>0}}\mc{P}_{\ul{d}}.\]
An alternative important realization of the specialization maps at the level of weight spaces arises as follows. The natural comultiplication map $D^{d_i+d_{i+1}} \lra D^{d_i}\oo D^{d_{i+1}}$ gives rise to maps
\begin{equation}\label{eq:comult-on-div-pows}
 \Delta_i : D^{\psi_i(\ul{d})} \lra D^{\ul{d}},\quad\text{ and }\quad \Hom(\Delta_i,\mc{P}) : \Hom(D^{\ul{d}},\mc{P}) \lra \Hom(D^{\psi_i(\ul{d})},\mc{P}).
\end{equation}
Using \eqref{eq:canonical-wtspace=Hom}, we get a canonical identification $\psi_i = \Hom(\Delta_i,\mc{P})$.

There is a decreasing filtration by subcomplexes
\begin{equation}\label{eq:filtr-Pfull-complex}
 \mc{P}(\kk^{\bullet})^{full} = \mc{F}^1_{\bullet} \supseteq \cdots \supseteq \mc{F}^d_{\bullet} \supseteq 0,
\end{equation}
where $\mc{F}^a_{\bullet} = \mc{F}^a_{\bullet}(\mc{P})$ is given by
\[ \mc{F}^a_n = \bigoplus_{\substack{\ul{d}\in\bb{Z}^n_{>0} \\ d_1\geq a}}\mc{P}_{\ul{d}}.\]
The associated graded complexes $\mf{gr}\mc{F}^a_{\bullet}=\mc{F}^a_{\bullet} / \mc{F}^{a+1}_{\bullet}$ induced by \eqref{eq:filtr-Pfull-complex} are given by
\begin{equation}\label{eq:grFan-weights}
 \mf{gr}\mc{F}^a_n = \bigoplus_{\substack{\ul{d}\in\bb{Z}^n_{>0} \\ d_1= a}}\mc{P}_{\ul{d}}.
\end{equation}
Using \eqref{eq:weight-shift}, we get an identification
\begin{equation}\label{eq:grFP-full-Pa}
 \mf{gr}\mc{F}^a_{\bullet+1} = \mc{P}^{(a)}(\kk^{\bullet})^{full}.
\end{equation}

\begin{theorem}\label{thm:Ext-from-spec-complexes}
  Suppose that $\mc{P}$ is a polynomial functor of degree $d=a+b$, and consider the associated complexes $\mc{F}^a_{\bullet}$ and $\mf{gr}\mc{F}^a_{\bullet}$ as above. We have for each cohomological degree $i$ an identification:
  \begin{enumerate}
   \item For the hook partition $\mu=(a,1^b)$,
   \[\Ext^i\left(\bb{W}_{\mu},\mc{P} \right) = H_{b+1-i}\left(\mc{F}^a_{\bullet}\right).\]
   \item $\Ext^i\left(D^a \oo \bw^b,\mc{P} \right) = H_{b+1-i}\left(\mf{gr}\mc{F}^a_{\bullet}\right) = \Ext^i\left(\bw^b,\mc{P}^{(a)} \right) $.
   \item Via the identifications (1), (2), the homology long exact sequence associated to the short exact sequence
   \[ 0 \lra \mc{F}^{a+1}_{\bullet} \lra \mc{F}^a_{\bullet} \lra \mf{gr}\mc{F}^a_{\bullet}\lra 0\]
   agrees up to a shift with the $\Ext(-,\mc{P})$ long exact sequence associated to the short exact sequence
   \[ 0 \lra \bb{W}_{(a+1,1^{b-1})} \lra D^a \oo \bw^b \lra \bb{W}_{\mu} \lra 0.\]
   More precisely, we have for each $i$ a commutative diagram
   \[
   \xymatrix{
   H_{b+1-i}(\mc{F}^{a+1}_{\bullet}) \ar@{=}[d] \ar[r] & H_{b+1-i}(\mc{F}^{a}_{\bullet}) \ar@{=}[d] \ar[r] &H_{b+1-i}(\mf{gr}\mc{F}^{a}_{\bullet}) \ar@{=}[d] \ar[r] & H_{b-i}(\mc{F}^{a+1}_{\bullet}) \ar@{=}[d] \\
   \Ext^{i-1}(\bb{W}_{(a+1,1^{b-1})},\mc{P}) \ar[r] & \Ext^i(\bb{W}_{\mu},\mc{P}) \ar[r] & \Ext^i(D^a \oo \bw^b,\mc{P}) \ar[r] & \Ext^i(\bb{W}_{(a+1,1^{b-1})},\mc{P}) 
   }
   \]
  \end{enumerate}
\end{theorem}

\begin{proof} It follows from \cite{mal-resolutions}*{Theorem 1} that $\bb{W}_{\mu}$ admits a projective resolution $\cdots \lra P_1\lra P_0$ by tensor products of divided powers, where
\[ P_i = \bigoplus_{\substack{\ul{d}\in\bb{Z}^{b+1-i}_{>0} \\ d_1\geq a}}D^{\ul{d}},\]
with differentials induced by comultiplication maps as in \eqref{eq:comult-on-div-pows}. Using the fact that $\Hom(D^{\ul{d}},\mc{P}) = \mc{P}_{\ul{d}}$, and $\Hom(\Delta_i,\mc{P}) = \psi_i$, we get that $\Hom(P_{\bullet},\mc{P}) = \mc{F}^a_{b+1-\bullet}$, hence (1) holds.

It follows from the construction in \cite{mal-resolutions}*{Theorem 1} that we have a subcomplex $P'_{\bullet}$ and a quotient $\ol{P}_{\bullet}=P_{\bullet}/P'_{\bullet}$, with terms given by
\[ P'_i = \bigoplus_{\substack{\ul{d}\in\bb{Z}^{b+1-i}_{>0} \\ d_1=a}}D^{\ul{d}},\quad\quad \ol{P}_i = \bigoplus_{\substack{\ul{d}\in\bb{Z}^{b+1-i}_{>0} \\ d_1\geq a+1}}D^{\ul{d}},\]
and that moreover $P'_{\bullet}$ is a resolution of $D^a \oo \bw^b$, and that $\ol{P}_{\bullet+1}$ is a resolution of $\bb{W}_{(a+1,1^{b-1})}$. We get as before that
\[  \Hom(P'_{\bullet},\mc{P}) = \mf{gr}\mc{F}^a_{b+1-\bullet}\]
hence the first equality in (2) holds. The second equality in (2) follows from the identifications
\[ H_{b+1-i}\left(\mf{gr}\mc{F}^a_{\bullet}\right) = H_{b-i}\left(\mf{gr}\mc{F}^a_{\bullet+1}\right) \overset{\eqref{eq:grFP-full-Pa}}{=} H_{b-i}\left(\mc{P}^{(a)}(\kk^{\bullet})\right)= \Ext^i\left(\bw^b,\mc{P}^{(a)} \right),\]
where the last equality follows by applying (1) to the partition $\mu=(1,1^{b-1})$.

The inclusion $P'_{\bullet}\subseteq P_{\bullet}$ lifts the natural map $D^a\oo\bw^b \lra \bb{W}_{\mu}$, and the short exact sequence
\[ 0 \lra P'_{\bullet} \lra P_{\bullet} \lra \ol{P}_{\bullet}\lra 0\]
represents the exact triangle
\[ D^a\oo\bw^b \lra \bb{W}_{\mu} \lra \bb{W}_{(a+1,1^{b-1})}[1]\]
in the derived category of $\RG$. Using that $\Hom(\ol{P}_{\bullet+1},\mc{P}) = \mc{F}^{a+1}_{b-\bullet}$ we get the desired conclusion.
\end{proof}

\begin{example}\label{ex:W22-Ext}
 If we let $\mc{P}=\bb{W}_{(2,2)}$ and use the standard tableau basis for $\bb{W}_{(2,2)}(\kk^n)$, then $\mc{P}(\kk^{\bullet})^{full}$ takes the form $\mc{P}(\kk^{4})^{full}\lra\mc{P}(\kk^{3})^{full}\lra\mc{P}(\kk^{2})^{full}$:
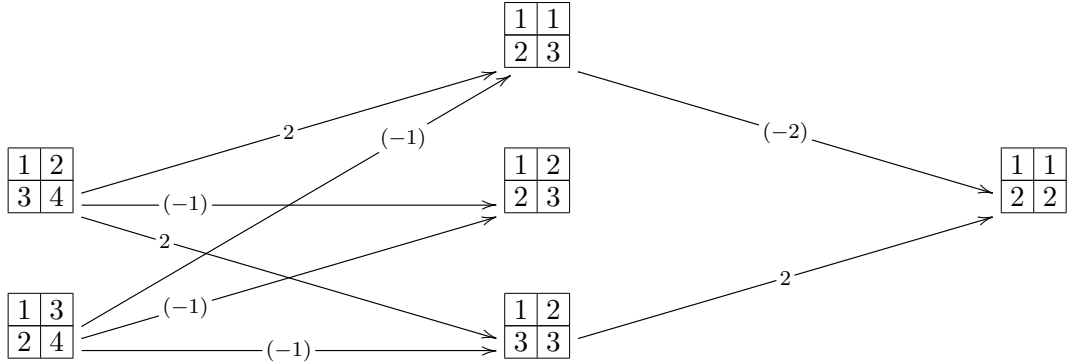
\[
\xymatrix@C=5.5cm{
 & {\young(11,23)} \ar[dr]|-{(-2)} & \\
 {\young(12,34)} \ar[ur]|-{2} \ar[r]|-(.25){(-1)} \ar[dr]|-(.2){2} & {\young(12,23)} & {\young(11,22)} \\
 {\young(13,24)} \ar[uur]|-(.75){(-1)} \ar[ur]|-(.25){(-1)} \ar[r]|-{(-1)}& {\young(12,33)} \ar[ur]|-{2} &  \\
}
\]
where for the differential $\mc{P}(\kk^{4})^{full}\lra\mc{P}(\kk^{3})^{full}$ we used the straightening relations
\[\Yvcentermath1
\psi_1\left(\young(13,24)\right)=\young(12,13) = -\young(11,23),\quad 
\psi_3\left(\young(13,24)\right)=\young(13,23) = -\young(12,33).
\]
We can then write the complex explicitly as
$$ \kk^2 \xra{\begin{pmatrix}
    2  & -1 \\ -1  & -1   \\ 2 & -1
\end{pmatrix}} \kk^3 \xra{\begin{pmatrix} -2 & 0 & 2 \end{pmatrix}} \kk^1,$$ 
which is homotopy equivalent to the complex
\begin{equation}\label{eq:Z121-complex}
\kk^1 \xra{\begin{pmatrix}
    3 \\ 3
\end{pmatrix}} \kk^2 \xra{\begin{pmatrix}
    -2 &2 
\end{pmatrix}} \kk^1.
\end{equation}
Notice that in the filtration \eqref{eq:filtr-Pfull-complex} we have $\mc{F}^3_{\bullet}=0$, and $\mc{F}^2_{\bullet}=\mf{gr}\mc{F}^2_{\bullet}$ is simply given by
\[\Yvcentermath1 
\xymatrix@C=3.5cm{
{\young(11,23)} \ar[r]|-{2} & {\young(11,22)}
}\qquad\qquad\text{or}\qquad\qquad \kk\xra{\times 2}\kk.
\]
Using Theorem~\ref{thm:Ext-from-spec-complexes}(1),(2) and the identification $\mc{P}^{(2)}$ as the skew Weyl functor $\bb{W}_{(2,2)/(2)}=\bb{W}_{(2)}=D^2$, it follows that the above complex computes any of the equivalent groups
\[ \Ext^{\bullet}\left(\bb{W}_{(2,1,1)},\bb{W}_{(2,2)}\right) = \Ext^{\bullet}\left(D^2\oo\bw^2,\bb{W}_{(2,2)}\right) = \Ext^{\bullet}\left(\bw^2,D^2\right).\]

If we take instead $\mc{P}=\bb{W}_{(3)}=D^3$ a divided power functor, then $\mc{P}(\kk^{\bullet})^{full}$ takes the form
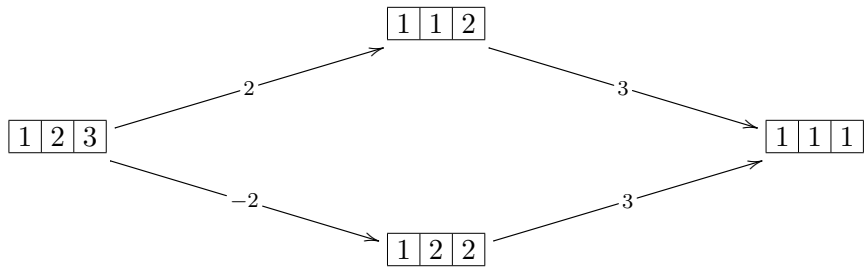
\[
\xymatrix@C=3.5cm{
& {\young(112)} \ar[dr]|-{3}& \\ 
{\young(123)} \ar[ur]|-{2} \ar[dr]|-{-2} & & {\young(111)} \\
& {\young(122)} \ar[ur]|-{3} & \\ 
}
\]
hence this complex is isomorphic (up to a homological shift) to the dual of \eqref{eq:Z121-complex}. For $\kk$ a field, this illustrates the special case $d=m=2$ of \eqref{eq:strange-duality} if we use the identifications $\Ext(\bb{S}_{\ll},\bw^d) = \Ext(\bw^d,\bb{W}_{\ll})$ where $d=|\ll|$.
\end{example}

In a similar way to \eqref{eq:spec-maps}, it will be useful to consider \defi{generization maps} $\psi^i:\kk^{n-1}\lra \kk^{n}$, defined via
\begin{equation}\label{eq:gener-maps} 
\psi^i(e_j) = \begin{cases}
e_j & \text{ if }j \leq i, \\
e_{j+1} & \text{ if }j \geq i+1.
\end{cases}
\end{equation}
The specialization and generization maps satisfy the simplicial identities \cite{weibel}*{Corollary~8.1.4}
\begin{equation}\label{eq:cosimplicial-psi}
 \begin{aligned}
 	\psi^j\psi^i & = \psi^i\psi^{j-1}\quad\text{ if }i<j, \\
 	\psi_j\psi_i & = \psi_i\psi_{j+1}\quad\text{ if }i\leq j, \\
	\psi_j\psi^i &=\begin{cases}
		\psi^i\psi_{j-1} & \text{ if }i<j-1, \\
		\text{identity} & \text{ if }i=j-1\text{ or }i=j,\\
		\psi^{i-1}\psi_j & \text{ if }i\geq j+1.
	\end{cases}
 \end{aligned}
\end{equation}
One can check that \eqref{eq:Pspec-exact} is homotopically trivial, with a nullhomotopy provided by the generization map~$\psi^0$.

It will be useful to consider another subcomplex $\mc{P}(\kk^{\bullet})^{ext}$ of \eqref{eq:Pspec-exact}, defined by
\[ \mc{P}(\kk^n)^{ext} = \bigoplus_{\substack{\ul{d}\in\bb{Z}^n_{\geq 0} \\ d_1,d_n>0}}\mc{P}_{\ul{d}}\]
and note that $\mc{P}(\kk^{\bullet})^{full}$ is a subcomplex of $\mc{P}(\kk^{\bullet})^{ext}$. In other words, the complex $\mc{P}(\kk^n)^{ext}$ removes the full support condition on tableaux and only insists that the first and last entries of the weights $\ul{d}$ are positive; note that as a result, $\mc{P}(\kk^{\bullet})^{ext}$ is generally a complex of infinite length.

We define in analogy to \eqref{eq:filtr-Pfull-complex} a filtration
\begin{equation}\label{eq:filtr-Pext-complex}
 \mc{P}(\kk^{\bullet})^{ext} = \hat{\mc{F}}^1_{\bullet} \supseteq \cdots \supseteq \hat{\mc{F}}^d_{\bullet} \supseteq 0,
\end{equation}
where $\hat{\mc{F}}^a_{\bullet} = \hat{\mc{F}}^a_{\bullet}(\mc{P})$ is given by
\[ \hat{\mc{F}}^a_n = \bigoplus_{\substack{\ul{d}\in\bb{Z}^n_{\geq 0} \\ d_1\geq a,\ d_n>0}}\mc{P}_{\ul{d}}.\]
While most of the time the complexes $\mc{F}^a_{\bullet}$ will be sufficient for our purposes, we will need to refer to the extended complexes $\hat{\mc{F}}^a_{\bullet}$ in Section~\ref{subsec:twisted-Koszul}. The following shows that they encode the same homological information.

\begin{lemma}\label{lem:Fdegenerate}
 There exist acyclic subcomplexes $\mc{D}^a_{\bullet}$ of $\hat{\mc{F}}^a_{\bullet}$ with terms given by
 \begin{equation}\label{eq:def-D-a-n} 
 \mc{D}^a_{n} = \bigoplus_{\substack{\ul{d}\in\bb{Z}^n_{\geq 0} \\ d_1\geq a,\ d_n>0 \\ d_i=0\text{ for some }i}}\mc{P}_{\ul{d}}
 \end{equation}
 In particular we have $\hat{\mc{F}}^a_{\bullet} = \mc{F}^a_{\bullet} \oplus \mc{D}^a_{\bullet}$ and the inclusion $\mc{F}^a_{\bullet} \subseteq \hat{\mc{F}}^a_{\bullet}$ is a quasi-isomorphism.
\end{lemma}

\begin{proof} We construct an augmented simplicial set $A_{\bullet}$ (see \cite{weibel}*{Proposition~8.1.3, and Augmented Objects 8.4.6}), where $A_{n-2}=\hat{\mc{F}}^a_n$ for $n\geq 1$. The augmentation $\epsilon:A_0\lra A_{-1}$ is the specialization $\psi_1: \hat{\mc{F}}^a_2 \lra \hat{\mc{F}}^a_1$, and the face operators $\partial_i : A_n \lra A_{n-1}$ ($n\geq 1$) and degeneracy operators $\sigma_i:A_n\lra A_{n+1}$ ($n\geq 0$) are
\[ \partial_i = \psi_{i+1}\quad\text{ and }\quad \sigma_i = \psi^{i+1} \quad\text{ for }i=0,1,\cdots,n.\]
It follows that the terms \eqref{eq:def-D-a-n} are the subspaces generated by the image of the degeneracies $\sigma_i$, and therefore they form an acyclic subcomplex by the proof of \cite{weibel}*{Theorem~8.3.8}. By construction we have $\hat{\mc{F}}^a_{\bullet} = \mc{F}^a_{\bullet} \oplus \mc{D}^a_{\bullet}$ hence the inclusion $\mc{F}^a_{\bullet} \subseteq \hat{\mc{F}}^a_{\bullet}$ is a quasi-isomorphism.
\end{proof}

\section{Stable cohomology and Koszul--Ringel duality}
\label{sec:stcoh-revisit}

The goal of this section is to revisit our prior work on stable cohomology \cite{RV} in order to extend it over an arbitrary commutative base ring~$\kk$, and to explain its close relation to Koszul--Ringel duality as developed by Krause \cite{Krause-duality} (following \cites{chalup-Adv,touze}). We fix the category $\RG$ of polynomial functors of degree $d$ over a commutative ring $\kk$, as in Section~\ref{subsec:poly-fun}. We consider a projective space $\PP=\PP^{N}_{\kk}$ where $N\geq d$, as well as the corresponding Euler sequence
\begin{equation}\label{eq:Euler-sequence} 
0 \lra \Omega \lra \mc{O}_\PP(-1)^{\oplus(N+1)} \lra \mc{O}_\PP \lra 0,
\end{equation}
where $\Omega$ denotes the cotangent sheaf on $\PP$. We define a functor $\Xi:\RG\lra \RG$ via
\begin{equation}\label{eq:def-Xi-RG}
 \Xi(\mc{P})(-) = H^d\left(\PP,\mc{P}(\Omega\oo -)\right)
\end{equation}
We will show that $\Xi$ is independent of the dimension $N\geq d$ of $\PP$, and that it coincides with the functor $\bw^d\oo_{\Gamma^d_{\kk}}-$ in \cite{Krause-duality}.

\subsection{Basics on (stable) cohomology}
\label{subsec:stcoh-basics}

We start by recalling from \cite{hartshorne}*{Theorem~III.5.1} that
\begin{equation}\label{eq:HO-i=0}
 H^j\left(\PP,\mc{O}_{\PP}(-i)\right) = 0\text{ for all $j$ and all }i=1,\cdots,N,
\end{equation}
and that the only non-vanishing cohomology for $\mc{O}_{\PP}$ is $H^0\left(\PP,\mc{O}_{\PP}\right)=\kk$. We write $\Omega^i=\bw^i\Omega$ and obtain from \cite{hartshorne}*{Exercise~II.5.16} and \eqref{eq:Euler-sequence} short exact sequences
\[ 0 \lra \Omega^i \lra \mc{O}_\PP(-i)^{\oplus{N+1\choose i}} \lra \Omega^{i-1}\lra 0,\]
from which we conclude by induction that if $i\leq N$ (and in particular for $i=d$)
\begin{equation}\label{eq:coh-Omegai}
 H^i\left(\PP,\Omega^i\right)=\kk,\quad H^j\left(\PP,\Omega^i\right)= 0\text{ for }j\neq i.
\end{equation}
Notice that $\Omega^d=\bb{W}_{\mu}\Omega$ where $\mu=(1,\cdots,1)=(1^d)$. For the remaining partitions of $d$ we have the following (see also \cite{RV}*{Theorem~4.5}).

\begin{lemma}\label{lem:Weyl-Omega}
    If $\mu$ is a partition of $d$ with $\mu_1\geq 2$, and if $N=\dim\PP \geq d$, then
    \begin{equation}\label{eq:Hj-WeylOmega=0} 
    H^j\left(\PP,\bb{W}_{\mu}\Omega\right) = 0\text{ for all }j.
    \end{equation}
\end{lemma}

\begin{proof} As explained in \cite{RV}*{Section~4.6}, it follows from \cite{ABW}*{Corollary~V.1.15} that there exists a right resolution of $\bb{W}_{\mu}\Omega$ given by the complex $\mc{G}^{\bullet}$, where the non-zero terms are
\begin{equation}\label{eq:def-Gi}
 \mc{G}^i = \bb{W}_{\mu/(1^i)}\left(\kk^{N+1}\right) \oo \mc{O}_{\PP}(-|\mu|+i)\quad \text{ for }i=0,\cdots,\mu'_1.
\end{equation}
Since $\mu_1\geq 2$, it follows that $|\mu|>\mu'_1$ and therefore
\[ -N \leq -d \leq -|\mu|+i < 0\quad\text{ if }0\leq i\leq \mu'_1.\]
We get using \eqref{eq:HO-i=0} that each $\mc{G}^i$ has vanishing cohomology, hence the same is true about $\bb{W}_{\mu}\Omega$.
\end{proof}

\begin{corollary}\label{cor:kunneth-divided}
 Let $\ul{d}\in\bb{Z}^n_{>0}$ with $d_1+\cdots+d_n=d$, and assume $N=\dim\PP \geq d$. If $\ul{d}\neq(1^d)$ then
 \[ H^j\left(\PP,D^{\ul{d}}\Omega\right)=0\quad\text{ for all }j.\]
 If $\ul{d}=(1^d)$ then $D^{\ul{d}}\Omega=\Omega^{\oo d}$ and we have a natural action of the symmetric group $\mf{S}_d$ by permuting the factors. Then $H^d\left(\PP,\Omega^{\oo d}\right)=\kk$ is isomorphic to the sign representation of $\mf{S}_d$, and $H^j\left(\PP,\Omega^{\oo d}\right)=0$ for $j\neq d$.
\end{corollary}

\begin{proof} It follows from Pieri's rule and induction (see also \cite{Krause-book}*{Theorem~8.4.11}, \cite{kou}*{Theorem~2.6}) that there exists a filtration of $D^{\ul{d}}\Omega$ with composition factors $\bb{W}_{\mu}\Omega$, where $\mu_1\geq\max(d_i)$. It follows from Lemma~\ref{lem:Weyl-Omega} that if $\ul{d}\neq(1^d)$ then $D^{\ul{d}}\Omega$ has vanishing cohomology. 

Suppose now that $\ul{d}=(1^d)$, and consider the exterior multiplication map
\begin{equation}\label{eq:ext-mult-Omega}
\Omega^{\oo d} \lra \Omega^d.
\end{equation}
This map is surjective, and similarly to the previous paragraph, the kernel has a filtration with composition factors $\bb{W}_{\mu}\Omega$, where $\mu_1\geq 2$. It follows from Lemma~\ref{lem:Weyl-Omega} that \eqref{eq:ext-mult-Omega} induces an isomorphism in cohomology, hence $H^d\left(\PP,\Omega^{\oo d}\right)=\kk$ and $H^j\left(\PP,\Omega^{\oo d}\right)=0$ for $j\neq d$. The map \eqref{eq:ext-mult-Omega} is $\mf{S}_d$-equivariant if we let $\sigma\in\mf{S}_d$ act on $\Omega^d$ as scalar multiplication by $\op{sgn}(\sigma)$. The induced map in cohomology is therefore an isomorphism of $\mf{S}_d$-representations. Since scalar multiplication on a sheaf induces multiplication by the same scalar in cohomology, $H^d\left(\PP,\Omega^d\right)=\kk$ is the sign representation of $\mf{S}_d$, and the desired conclusion follows.
\end{proof}

As a consequence of the constructions above we get the following stable cohomology functors.

\begin{theorem}\label{thm:Hjst-to-Modk}
 For each $j$ there is a \defi{stable cohomology functor}
 \[ H^j_{st}:\RG\lra \Mk,\quad\text{defined by}\quad H^j_{st}(\mc{P})=H^j_{st}(\mc{P}\Omega) = H^j\left(\PP,\mc{P}\Omega\right),\]
 which is independent of $N=\dim\PP \geq d$. Moreover, we have that
 \begin{itemize}
  \item $H^j_{st}\equiv 0$ for $j>d$ (and for $j<0$).
  \item $H^d_{st}$ is right exact, and $H^{d-j}_{st}$ are the left derived functors of $H^d_{st}$.
 \end{itemize}
\end{theorem}

\begin{proof} It is clear that $H^j_{st}$ defines a functor, and to prove independence of $N$ we recall from the proof of \cite{Krause-duality}*{Theorem~2.10} that the functors $D^{\ul{d}}$ form a family of projective generators of $\RG$. It follows that each $\mc{P}$ has a projective resolution $\mc{F}_{\bullet}:\cdots\lra\mc{F}_1\lra\mc{F}_0$, where $\mc{F}_i$ is a direct sum of $D^{\ul{d}}$. It follows that $\mc{F}_\bullet(\Omega)$ is a resolution of $\mc{P}(\Omega)$, and by Corollary~\ref{cor:kunneth-divided} each $\mc{F}_i(\Omega)$ can only have cohomology in degree $d$, which is independent of $N$. The complex $H^d\left(\PP,\mc{F}_\bullet(\Omega)\right)$ is then independent of $N$ and we have
\begin{equation}\label{eq:HjstP=H-HdFOmega}
 H^j_{st}(\mc{P}) = H_{d-j}\left(H^d\left(\PP,\mc{F}_\bullet(\Omega)\right)\right).
\end{equation}
Since $\mc{F}_{\bullet}$ is concentrated in non-negative degrees, this shows that $H^j_{st}\equiv 0$ for $j>d$. The remaining conclusions follow using the long exact sequence in sheaf cohomology.
\end{proof}

\begin{example}\label{ex:Hst-Sym2}
 Consider the polynomial functor $\Sym^2$ and the short exact sequence
 \[ 0 \lra \Omega^2 \lra \Omega\oo\Omega \lra \Sym^2\Omega \lra 0.\]
 The long exact sequence in cohomology, together with \eqref{eq:coh-Omegai}, Corollary~\ref{cor:kunneth-divided}, and \cite{RV}*{Corollary~4.9}, yields the following exact sequence
 \[ 0 \lra H^1_{st}(\Sym^2) \lra \kk \overset{\times 2}{\lra} \kk \lra H^2_{st}(\Sym^2) \lra 0\]
 This shows that $H^1_{st}(\Sym^2) = (0:_{\kk} 2)$ is the $2$-torsion in $\kk$, while $H^2_{st}(\Sym^2) = \kk/2\kk$. When $\kk$ is a field of characteristic $2$, this yields $H^1_{st}(\Sym^2)=H^2_{st}(\Sym^2)=\kk$ (see also \cite{RV}*{Example~6.11}), while for $\kk=\bb{Z}$ we get $H^2_{st}(\Sym^2)=\bb{Z}/2\bb{Z}$ and the other stable cohomology groups vanish.
\end{example}

We can think of Corollary~\ref{cor:kunneth-divided} as a (trivial version of the) K\"unneth formula for tensor products of divided powers. It follows that if $d=d'+d''$, $\mc{F}_\bullet$ (resp. $\mc{G}_\bullet$) are complexes whose terms are direct sums of $D^{\ul{d}'}$, $|\ul{d}'|=d'$ (resp. $D^{\ul{d}''}$, $|\ul{d}''|=d''$), then we have an isomorphism of complexes
\begin{equation}\label{eq:tensor-FG-stcoh}
H^{d'}\left(\PP,\mc{F}_\bullet(\Omega)\right) \oo H^{d''}\left(\PP,\mc{G}_\bullet(\Omega)\right) = H^d\left(\PP,\mc{F}_\bullet(\Omega)\oo\mc{G}_\bullet(\Omega)\right).
\end{equation}
The K\"unneth formula \cite{RV}*{Theorem~4.3} then follows in general when $\kk$ a field. For a more general ring $\kk$ we have the following consequences that will be used later.

\begin{proposition}\label{prop:kunneth} 
 If $\mc{P}'$, $\mc{P}''$ are polynomial functors of degrees $d',d''$, with $d'+d''=d$, then
 \begin{equation}\label{eq:Hdst-PotimesP}
  H^d_{st}(\mc{P}'\oo\mc{P}'') = H^{d'}_{st}(\mc{P}') \oo H^{d''}_{st}(\mc{P}'').
 \end{equation}
 If moreover $H^{j'}_{st}(\mc{P}')=H^{j''}_{st}(\mc{P}'')=0$ for $j'<d'$, $j''<d''$ then
 \begin{equation}\label{eq:Hjst-TorPP}
 H^{d-j}_{st}(\mc{P}'\oo\mc{P}'') = \Tor_j^{\kk}\left(H^{d'}_{st}(\mc{P}'), H^{d''}_{st}(\mc{P}'') \right)\quad\text{for all }j.
 \end{equation}
\end{proposition}

\begin{proof}
 Let $\mc{F}_\bullet$ (resp. $\mc{G}_\bullet$) be resolutions of $\mc{P}'$ (resp. $\mc{P}''$) as in the proof of Theorem~\ref{thm:Hjst-to-Modk}. Conclusion \eqref{eq:Hdst-PotimesP} follows from \eqref{eq:HjstP=H-HdFOmega}, \eqref{eq:tensor-FG-stcoh} and the right exactness of tensor products. For the second part, the vanishing assumptions on stable cohomology imply that $H^{d'}\left(\PP,\mc{F}_\bullet(\Omega)\right)$ is a free resolution of $H^{d'}_{st}(\mc{P}')$ over $\kk$, and that $H^{d''}\left(\PP,\mc{G}_\bullet(\Omega)\right)$ is a free resolution of $H^{d''}_{st}(\mc{P}'')$, so \eqref{eq:Hjst-TorPP} follows again from \eqref{eq:tensor-FG-stcoh}.
\end{proof}

\begin{example}\label{ex:stcoh-Sym2ooSym2}
 Building on Example~\ref{ex:Hst-Sym2}, K\"unneth's formula implies that if $\kk$ is a field of characteristic $2$ then
 \[ H^2_{st}(\Sym^2\oo\Sym^2) = H^4_{st}(\Sym^2\oo\Sym^2) = \kk,\quad H^3_{st}(\Sym^2\oo\Sym^2) = \kk^{\oplus 2},\]
 while for $\kk=\bb{Z}$ it follows from Proposition~\ref{prop:kunneth} that
 \[ H^3_{st}(\Sym^2\oo\Sym^2) = H^4_{st}(\Sym^2\oo\Sym^2) = \bb{Z}/2\bb{Z},\]
 and the other stable cohomology groups vanish.
\end{example}

\subsection{The stable cohomology functor on $\RG$ and Koszul--Ringel duality}
\label{subsec:stcoh-RG}

We can enhance stable cohomology to a collection of functors on the category of polynomial representations. We define for each $j$
\begin{equation}\label{eq:def-HHst}
 \HH^j_{st}:\RG\lra \RG,\quad \HH^j_{st}(\mc{P})(V)=H^j_{st}\left(\mc{P}(\Omega\oo V)\right)\quad\text{for }\mc{P}\in\RG,\ V\in\mc{V}_{\kk}.
\end{equation}
It follows from Theorem~\ref{thm:Hjst-to-Modk} that $\HH^j_{st}\equiv 0$ for $j>d$ (and for $j<0$), that $\HH^d_{st}$ is right exact, and that $\HH^{d-j}_{st}$ are the left derived functors of $\HH^d_{st}$. We start by analyzing these functors for $\mc{P}$ a tensor product of divided powers.

\begin{proposition}\label{prop:stcoh-transpose-Dcomposition}
 If $|\ul{d}|=d$ then we have
 \[ \HH^j_{st}\left(D^{\ul{d}}\right) = 0\quad\text{ for all }j\neq d,\text{ and }\quad \HH^d_{st}\left(D^{\ul{d}}\right) = \bw^{\ul{d}}.\]
\end{proposition}

\begin{proof} We first discuss the case when $\ul{d}=(d)$ is a singleton. Using the Cauchy filtration for divided powers \cite{Krause-book}*{Theorem~8.4.6}, we have a short exact sequence
\begin{equation}\label{eq:ses-Dd-to-wedge-twice}
 0 \lra \mc{K} \lra D^d(\Omega\oo V) \overset{\pi}{\lra} \Omega^d \oo \bw^d V \lra 0,
\end{equation}
where $\mc{K}$ has a filtration with composition factors of the form
\[ \bb{W}_{\ll}\Omega \oo \bb{W}_{\ll}V,\text{ where }\ll\vdash d\text{ with }\ll_1\geq 2.\]
It follows from Lemma~\ref{lem:Weyl-Omega} that $\mc{K}$ has vanishing cohomology, hence we get natural isomorphisms
\[H^j_{st}\left(D^{d}(\Omega\oo V)\right) = H^j_{st}\left(\Omega^d \oo \bw^d V\right) = H^j_{st}\left(\Omega^d\right) \oo \bw^d V\quad\text{ for all }j.\]
The desired description for $\HH^j_{st}(D^d)$ follows from \eqref{eq:coh-Omegai}.

For the general $\ul{d}$ we know that $H^j_{st}(D^{d_i}(\Omega\oo V))=0$ for $j\neq d_i$, and $H^{d_i}_{st}(D^{d_i}(\Omega\oo V))=\bw^{d_i}V$ is a projective module, hence flat. The desired conclusion follows from Proposition~\ref{prop:kunneth}.
\end{proof}

We next consider the action of stable cohomology on maps between tensor products of divided powers. Following \cite{Krause-book}*{Section~8.5}, for a composition $\ul{d}$ of $d$ we consider the multiplication and comultiplication transformations
\[ \nabla: T^d \lra D^{\ul{d}},\quad\hat{\nabla}:T^d \lra \bw^{\ul{d}},\quad\Delta:D^{\ul{d}} \lra T^d,\quad\hat{\Delta}:\bw^{\ul{d}} \lra T^d. \]
We will abuse notation and use the same symbols regardless of the composition $\ul{d}$.

\begin{lemma}\label{lem:basic-co-mult-stable}
 The divided power multiplication $\nabla:T^d \lra D^d$ induces the exterior multiplication on stable cohomology
   \[ \HH^d_{st}\left(T^d\right) = T^d \overset{\HH^d_{st}(\nabla)=\hat{\nabla}}{\lra} \HH^d_{st}\left(D^{d}\right) = \bw^d.\]
 Similarly, for the comultiplication $\Delta:D^d\lra T^d$ we have $\HH^d_{st}(\Delta) = \hat{\Delta}:\bw^d\lra T^d$.
\end{lemma}

\begin{proof} The identifications $\HH^d_{st}\left(T^d\right) = T^d$ and $\HH^d_{st}\left(D^{d}\right) = \bw^d$ follow by applying Proposition~\ref{prop:stcoh-transpose-Dcomposition} for $\ul{d}=(1^d)$ and $\ul{d}=(d)$ respectively. For multiplication, we evaluate the polynomial functors on a projective module $V$ and note that we have a commutative diagram
\[ 
\xymatrix{
T^d(\Omega \oo V) \ar[r]^{\nabla} \ar[dr]_{\hat{\nabla} \oo \hat{\nabla}} & D^d(\Omega\oo V) \ar[d]^{\pi} \\
& \Omega^d \oo \bw^d V
}
\]
where $\pi$ is the map from \eqref{eq:ses-Dd-to-wedge-twice}, and to define $\hat{\nabla} \oo \hat{\nabla}$ we use the identification $T^d(\Omega \oo V) = T^d(\Omega) \oo T^d(V)$. Since both $\pi$ and $\hat{\nabla}:T^d\Omega \lra \Omega^d$ induce isomorphisms in cohomology, we conclude that $H^d_{st}(\nabla)=\hat{\nabla}$.

For comultiplication, we write $\Psi = \HH^d_{st}(\Delta)$, and write $\Psi_V$ for its evaluation on a vector space $V$. To compute $\Psi_V(v_1\cdots v_d)$, we may assume without loss of generality that $V$ is free and $v_1,\cdots,v_d$ is a basis of $V$. We have
\[ D^d(\Omega\oo V) = D^d(\Omega v_1 \oplus \cdots \oplus \Omega v_d) =  \Omega^{\oo d} v_1\cdots v_d \bigoplus \mc{C},\]
where $\mc{C}$ is a direct sum of $D^{\ul{d}}\Omega$, where $\ul{d}$ is a composition of $d$ with some $d_i\geq 2$, and hence has vanishing stable cohomology by Corollary~\ref{cor:kunneth-divided}. Similarly, we have
\[ T^d(\Omega\oo V) = T^d(\Omega v_1 \oplus \cdots \oplus \Omega v_d) = \bigoplus_{\sigma\in\mf{S}_d} \Omega^{\oo d} v_{\sigma(1)}\oo\cdots\oo v_{\sigma(d)} \bigoplus \mc{C}',\]
where $\mc{C}'$ is a direct sum of copies of $\Omega^{\oo d}$ indexed by tensors $v_{i_1}\oo\cdots\oo v_{i_d}$ where some index $i_j$ appears at least twice. It follows that $\Delta$ restricts to a map
\[ \Omega^{\oo d} v_1\cdots v_d \lra \bigoplus_{\sigma\in\mf{S}_d} \Omega^{\oo d} v_{\sigma(1)}\oo\cdots\oo v_{\sigma(d)}\]
where each component $\Omega^{\oo d}\lra\Omega^{\oo d}$ is given by permuting the factors using the corresponding~$\sigma\in\mf{S}_d$. Using the last part of Corollary~\ref{cor:kunneth-divided} and $H^d_{st}(\Omega^d) = \kk$, we get that 
\[ \Psi_V(v_1\cdots v_d) = \sum_{\sigma\in\mf{S}_d} \op{sgn}(\sigma) v_{\sigma(1)}\oo\cdots\oo v_{\sigma(d)},\]
that is, $\Psi_V = \hat{\Delta}$ as desired.
\end{proof}

\begin{proposition}\label{prop:stcoh-transpose-maps}
  Given compositions $\ul{a},\ul{b}$ of $d$, and transformation $\phi:D^{\ul{a}} \lra D^{\ul{b}}$, the induced transformation
  \[ H^d_{st}\left(D^{\ul{a}}(\Omega\oo -)\right) \overset{H^d_{st}(\phi)}{\lra} H^d_{st}\left(D^{\ul{b}}(\Omega\oo -)\right)\]
  coincides with $\Xi(\phi)$.
\end{proposition}

\begin{proof} The action of $\Xi$ on morphisms between the functors $D^{\ul{d}}$ is characterized by the commutative diagram \cite{Krause-book}*{Proposition~8.5.3}
\[
\xymatrix{
\Hom(D^{\ul{a}},D^{\ul{b}}) \ar[r]^{\Xi} \ar[d]^{(\nabla,\Delta)} & \Hom(\bw^{\ul{a}},\bw^{\ul{b}}) \ar[d]^{(\hat{\nabla},\hat{\Delta})} \\
\Hom(T^d,T^d) \ar[r]^{\omega} & \Hom(T^d,T^d) 
}
\]
where $\Hom(T^d,T^d)$ is identified with $\kk[\mf{S}_d]^{op}$, and $\omega(\sigma) = \sgn(\sigma)\cdot\sigma$ for all $\sigma\in\mf{S}_d$. It follows that if we denote by $\psi$ the composition
\[ T^d \overset{\nabla}{\lra} D^{\ul{a}} \overset{\phi}{\lra} D^{\ul{b}} \overset{\Delta}{\lra} T^d\]
then 
\[\omega(\psi) = \hat{\Delta} \circ \Xi(\phi) \circ \hat{\nabla}.\]
Note also that by Corollary~\ref{cor:kunneth-divided}, we have $\HH^d_{st}(\psi) = \omega(\psi)$. Combining functoriality of stable cohomology with Lemma~\ref{lem:basic-co-mult-stable} we get that $H^d_{st}(\psi)$ is given by the composition
\[\HH^d_{st}(\psi) = \hat{\Delta} \circ \HH^d_{st}(\phi) \circ \hat{\nabla}.\]
It follows that $\HH^d_{st}(\phi)=\Xi(\phi)$, as desired.
\end{proof}

We are now ready to prove the main result of the section.

\begin{proof}[Proof of Theorem~\ref{thm:stcoh-transpose-duality}]
 We consider a projective resolution $(\mc{F}_{\bullet},\pd_i)$ of $\mc{P}$, where 
 \[ \mc{F}_j = \bigoplus_k D^{\ul{d}(j,k)}\text{ for some finite collection of compositions }\ul{d}(j,k)\text{ of d}.\]
 It follows from Proposition~\ref{prop:stcoh-transpose-Dcomposition} that $\mc{F}_j(\Omega\oo V)$ has cohomology concentrated in degree $d$, hence by the hypercohomology spectral sequence, the complex
 \[ G_{\bullet}(V) = H^d_{st}\left(\mc{F}_{\bullet}(\Omega \oo V)\right)\]
 computes stable cohomology for $\mc{P}(\Omega\oo V)$: we have
 \[ H_j(G_{\bullet}(V)) = H^{d-j}_{st}\left(\mc{P}(\Omega \oo V)\right) \quad \text{ for all }j.\]
 To verify the theorem, it is then enough to show that $G_{\bullet}(V)$ is naturally isomorphic to $\Xi(\mc{F}_{\bullet})(V)$, whose $j$-th homology computes $\LL_j\Xi(\mc{P})(V)$. By Proposition~\ref{prop:stcoh-transpose-Dcomposition}, the terms in $G_{\bullet}(V)$ are given by
 \[ G_j(V) = \bigoplus_k \bw^{\ul{d}(j,k)}V = \Xi(\mc{F}_j)(V),\]
 while the differentials are given by $\Xi(\partial_j)(V)$ by Proposition~\ref{prop:stcoh-transpose-maps}, concluding our proof.
\end{proof}

\subsection{The proof of  Theorem~\ref{thm:short-res-Schur}}
\label{subsec:short-res}

We let $r(\mu)=d-\mu_1$ and we prove conclusion (1) by induction on $r(\mu)$. If $r(\mu)=0$ then $\bb{W}_{\mu}=D^{\mu_1}$ and there is nothing to show. For the induction step, let $\ol{\mu}=(\mu_2,\mu_3,\cdots)$. There exists a short exact sequence
    \[ 0 \lra K \lra D^{\mu_1} \oo \bb{W}_{\ol{\mu}} \lra \bb{W}_{\mu} \lra 0,\]
    where, by Pieri's rule, $K$ has a filtration with composition factors of the form $\bb{W}_{\gamma}$ with $r(\gamma)<r(\mu)$. By induction, $K$ admits a divided power resolution of length $\leq r(\mu)-1$. Also by induction we have that $\bb{W}_{\ol{\mu}}$ has a resolution of length $\leq r(\mu)$. Since tensoring with $D^{\mu_1}$ preserves exactness, we conclude by the horseshoe lemma that $\bb{W}_{\mu}$ has the desired resolution of length $\leq r(\mu)$. 
    
To prove conclusion (2), we apply (1) to $\ll=\mu'$ to obtain a divided power resolution $\mc{F}_{\bullet}$ of $\bb{W}_{\ll}$ whose length is $\leq |\ll|-\ll_1=d-\ell$ (since $\ll_1=\ell(\mu)$). Applying the functor $\Xi$ to $\mc{F}_{\bullet}$ and using \cite{Krause-duality}*{Proposition~4.16}, it follows that $\Xi(\mc{F}_{\bullet})$ is a resolution of $\bb{S}_{\ll'}=\bb{S}_{\mu}$ by tensor products of exterior powers, as desired.

To prove (3), consider a resolution $\mc{G}_{\bullet}$ of $\bb{S}_{\mu}$ as in (2), so that $\mc{G}_j=0$ for $j>d-\ell$. Using \eqref{eq:coh-Omegai} and Proposition~\ref{prop:kunneth}, it follows that each term of $\mc{G}_{\bullet}$ has cohomology supported in degree $d$. It follows that
\[ H^j_{st}\left(\bb{S}_{\mu}\right) = H_{d-j}\left(H^d_{st}(\mc{G}_{\bullet})\right) = 0\text{ for }d-j>d-\ell.\]
This proves that $H^j_{st}\left(\bb{S}_{\mu}\right)=0$ for $j<\ell$, while the vanishing for $j>d$ comes from Theorem~\ref{thm:Hjst-to-Modk}.

To prove (4), we use the identification $\Ext^j\left(\bb{S}_{\mu},\bw^d\right)=\Ext^j\left(\bw^d,\bb{W}_{\mu}\right)$, which can then be calculated using Theorem~\ref{thm:Ext-from-spec-complexes}(1) as
\[ \Ext^j\left(\bw^d,\bb{W}_{\mu}\right) = H_{d-j}\left(\bb{W}_{\mu}(\kk^\bullet)^{full} \right).\]
Note that if $\ul{d}\in\bb{Z}^k_{>0}$, then a weight space $(\bb{W}_{\mu})_{\ul{d}}$ can only be non-zero if $k\geq\ell$. In particular we get $\Ext^j\left(\bw^d,\bb{W}_{\mu}\right)=0$ if $d-j<\ell$, or equivalently $j>d-\ell$, as desired.

\begin{example}\label{ex:res-W222}
 Consider the partition $\mu=(2,2,2)$. If we trace through the inductive argument for proving Theorem~\ref{thm:short-res-Schur}(1), it follows that $\bb{W}_{\mu}$ admits a projective resolution of length $4$, with terms given by
    $$\left(D^6\right)^{\oplus 2} \longrightarrow 
    \begin{matrix}
        \left(D^6\right)^{\oplus 3} \\ \oplus \\ \left(D^{(5,1)}\right)^{\oplus 2} \\ \oplus \\ D^{(4,2)} 
    \end{matrix} 
    \longrightarrow 
    \begin{matrix}
        D^6 \\ \oplus \\ \left(D^{(5,1)}\right)^{\oplus 2} \\ \oplus \\ \left(D^{(4,2)}\right)^{\oplus 2}\\ \oplus \\ D^{(3,3)} \\ \oplus \\ D^{(4,1,1)} 
    \end{matrix} 
    \longrightarrow 
    \begin{matrix}
        \left(D^{(4,2)}\right)^{\oplus 2} \\ \oplus \\ \left(D^{(3,2,1)}\right)^{\oplus 2}
    \end{matrix} 
    \longrightarrow 
    D^{(2,2,2)}.$$
 Note that in total one gets $20$ summands which are tensor products of divided powers. If instead one follows the method from \cite{AB2}*{Section~4} then the resulting resolution has length $5$ and a total of $24$ summands. The resolution constructed in \cite{san-yud}*{Theorem~5.2} has length $6$ and $412$ summands.
\end{example}

\subsection{Some consequences when $\kk$ is a field}
\label{subsec:stcoh-kfield}

In this section we assume that $\kk$ is a field of characteristic $p>0$, and discuss some quick, but important consequences of Theorem~\ref{thm:stcoh-transpose-duality}. We begin with the proof of the periodicity conjecture \cite{GRV}*{Conjecture~4.2} for stable cohomology.

\begin{proof}[Proof of Theorem~\ref{thm:periodicity}]
 We apply Theorem~\ref{thm:stcoh-transpose-duality} to the polynomial functor $\mc{P}=\bb{S}_{\mu}$, and evaluate both sides of the equality on the vector space $V=\kk$ to obtain
 \[ H^j_{st}\left(\bb{S}_{\mu}\right) = \LL_{d-j}\Xi(\bb{S}_{\mu})(\kk) = \Ext^{d-j}\left(\bb{S}_{\mu},\bw^d\right)^{\vee},\]
 where the last equality uses the description in \cite{chalup-Adv}*{Definition~2.3} for Koszul duality. Moreover, using \cite{chalup-Adv}*{Corollary~2.5}, this is further equal to
 \[  \Ext^{d-j}\left(\bb{W}_{\mu},D^d\right)^{\vee} = \Ext^{d-j}\left(\bb{W}_{\mu[q]},D^{d+q}\right)^{\vee}\]
 where the equality follows from \cite{periodicity-mal-ste}*{Theorem~1.1}. Translating back to stable cohomology, we get
 \[ \Ext^{d-j}\left(\bb{W}_{\mu[q]},D^{d+q}\right)^{\vee} = H^{j+q}_{st}\left(\bb{S}_{\mu[q]}\Omega\right)\]
 which yields the desired conclusion.
\end{proof}

Another consequence of Theorem~\ref{thm:stcoh-transpose-duality} employs the compatibility of stable cohomology and Frobenius \cite{RV}*{Theorem~4.4} to deduce the following (see also \cite{chalup-Adv}*{Proposition~2.6}, \cite{touze}*{Proposition~6.6}).

\begin{corollary}\label{cor:transpose-duality-Frobenius}
 If $\mc{P}$ has degree $d$, then for all $j\in \bb{Z}$ and all $q=p^k$ we have
 \[ \LL_{j+(q-1)d}\Xi\left(F^q\circ\mc{P}\right) = F^q\circ\left(\LL_j\Xi(\mc{P})\right).\]
\end{corollary}

\begin{proof} Since $F^q\circ\mc{P}$ has degree $dq$, we get from Theorem~\ref{thm:stcoh-transpose-duality} that
\[\LL_{j+(q-1)d}\Xi\left(F^q\circ\mc{P}\right)(V) = H^{d-j}_{st}\left((F^q\circ\mc{P})(\Omega\oo V)\right) = H^{d-j}_{st}\left(\mc{P}(F^q\Omega\oo F^qV)\right)\]
Applying \cite{RV}*{Theorem~4.4} to the functor $\mc{P}\left( - \oo F^qV\right)$, this is further equal to
\[ H^{d-j}_{st}\left(\mc{P}(\Omega\oo F^qV)\right) = F^q\left(H^{d-j}_{st}\left(\mc{P}(\Omega\oo V)\right) \right) = F^q(\LL_j\Xi(\mc{P})(V)),\]
where the last equality follows again by Theorem~\ref{thm:stcoh-transpose-duality}.
\end{proof}

\section{Extensions between Schur/Weyl functors associated to hooks}
\label{sec:hook-Ext}

The goal of this section is to prove Theorem~\ref{thm:introExtComputation} concerning the invariance property for extension groups between Schur functors associated to hooks. Using \eqref{eq:ext-weyl-schur=eqs}, this can be reformulated in terms of extensions between Weyl functors, which is most convenient for the structure of our argument. We work in the category $\RG$ where $\kk$ is any commutative ring, and we write $d=a+b=A+B$ with $A\geq a>0$, and $\Delta=A-a$. Based on \eqref{eq:ext-weyl-schur=eqs} and Theorem~\ref{thm:stcoh-transpose-duality},  Theorem~\ref{thm:introExtComputation} follows directly from the following. 

\begin{theorem}\label{thm:Ext-hook=Ext-divided}
 If $0\leq\Delta<A\leq d$ and if we set $B=d-A$, then
 \[ \Ext^i\left(\bb{W}_{(d-\Delta,1^{\Delta})},\bb{W}_{(d)} \right) = \Ext^i\left(\bb{W}_{(A-\Delta,1^{B+\Delta})},\bb{W}_{(A,1^B)} \right) \text{ for all }i.\]
\end{theorem}

\noindent{\bf The proof strategy.} We recall from Theorem~\ref{thm:Ext-from-spec-complexes} that the specialization complexes $\mc{F}^a_{\bullet}(\bb{W}_{(A,1^B)})$ provide explicit models for the computation of $\Ext^\bullet\left(\bb{W}_{(a,1^b)},\bb{W}_{(A,1^B)} \right)$. The most technical part of our argument is the construction of a morphism of complexes
 \begin{equation}\label{eq:Phi-on-Fa}
  \Phi:  \mc{F}^a_{\bullet}\left(\bb{W}_{(A,1^B)}\right) \lra \mc{F}^{a-1}_{\bullet+1}\left(\bb{W}_{(A-1,1^{B+1})}\right)
  \end{equation}
 which is based on the construction of some twisted analogues of Koszul differentials and is explained in Section~\ref{subsec:twisted-Koszul}. While $\Phi$ is in general not a quasi-isomorphism, it can be iterated to define morphisms of complexes
\begin{equation}\label{eq:def-PhiB-pre}
 \Phi^{[B]}:\mc{F}^{a+B}_{\bullet}(\bb{W}_{(d)}) \lra \mc{F}^a_{\bullet+B}(\bb{W}_{(A,1^B)})
\end{equation}
satisfying the divided power relations
\begin{equation}\label{eq:PhiB-div-power}
 \Phi\circ\Phi^{[B]} = (B+1)\cdot \Phi^{[B+1]},
\end{equation}
as explained in Section~\ref{subsec:div-pow-Phi}. The goal is then to prove that the divided power $\Phi^{[B]}$ is a quasi-isomorphism, which is done in Section~\ref{subsec:PhiB-quasi-isom} and yields Theorem~\ref{thm:Ext-hook=Ext-divided}. To do so we consider the filtrations \eqref{eq:filtr-Pfull-complex}, and reduce our problem to proving that the induced maps
 \[ \Phi^{[B]}:\mf{gr}\mc{F}^{a+B}_{\bullet}(\bb{W}_{(d)}) \lra \mf{gr}\mc{F}^a_{\bullet+B}(\bb{W}_{(A,1^B)})\]
are quasi-isomorphisms. For that we construct explicit retractions
  \[ \Theta^{[B]}:\mf{gr}\mc{F}^a_{\bullet+B}(\bb{W}_{(A,1^B)}) \lra \mf{gr}\mc{F}^{a+B}_{\bullet}(\bb{W}_{(d)}) \]
 which then induce surjective maps $H_n(\Theta^{[B]})$ for all $n$ (left inverses to $H_n(\Phi^{[B]})$). Applying Theorem~\ref{thm:Ext-from-spec-complexes}(2) together with a well-known degree reduction statement for $\Ext$ groups allows us to conclude that
 \[ H_n(\mf{gr}\mc{F}^a_{\bullet+B}(\bb{W}_{(A,1^B)})) \simeq \Ext^{\Delta+1-n}\left(\bw^{\Delta},D^{\Delta}\right)  \simeq H_n(\mf{gr}\mc{F}^{a+B}_{\bullet}(\bb{W}_{(d)}))\]
 are abstractly isomorphic. The surjections $H_n(\Theta^{[B]})$ must therefore be isomorphisms, which is enough to conclude that $\Phi^{[B]}$ is a quasi-isomorphism. We note that our constructions make sense over an arbitrary commutative ring $\kk$, and are compatible with base change. Therefore running the strategy above for $\kk=\bb{Z}$ is sufficient in order to deduce the conclusions over an arbitrary $\kk$. This observation will be crucial in verifying that $\Phi^{[B]}$ is a morphism of complexes at the end of Section~\ref{subsec:div-pow-Phi}, and won't be used otherwise.

\subsection{Twisted Koszul maps}
\label{subsec:twisted-Koszul}

The standard basis of $\kk^n$ gives rise to corresponding standard bases
\[ \{e^{\a} = e_1^{(\a_1)}\cdots e_n^{(\a_n)} : \a_1+\cdots+\a_n = A\} \quad\text{ for }\quad D^A(\kk^n),\]
\[\{e_{\b} = e_{\b_1}\wedge\cdots\wedge e_{\b_B} : 1\leq\b_1<\cdots<\b_B \leq n\} \quad\text{ for }\quad \bw^B\kk^n.\]
We define \defi{contraction operators}
\begin{equation}\label{eq:contraction-eta}
 \eta_j : D^A(\kk^n) \lra D^{A-1}(\kk^{n}),\quad\eta_j(e^{\a}) = \begin{cases}
  e_1^{(\a_1)}\cdots e_j^{(\a_j-1)}\cdots e_n^{(\a_n)} & \text{if }\a_j > 0, \\
  0 & \text{otherwise,}
 \end{cases}
\end{equation}
and extend them, in order to simplify the notation later on, to operators
\[\eta_j : \left(D^A\oo\bw^B\right)(\kk^n) \lra \left(D^{A-1}\oo\bw^B\right)(\kk^{n}),\quad \eta_j(e^{\a} \oo e_{\b})=\eta_j(e^{\a})\oo e_{\b}.\]
Although contraction operators make sense for $\bw^B$ as well, we will not employ them hence the notation above should not be ambiguous. We will use the exterior multiplication maps $-\wedge e_j:\bw^B(\kk^n) \lra \bw^{B+1}(\kk^n)$ and extend them naturally to maps $-\wedge e_j:(D^A\oo\bw^B)(\kk^n) \lra (D^A\oo\bw^{B+1})(\kk^n)$. The contraction operators satisfy the following compatibility relations with generization maps \eqref{eq:gener-maps} and specialization maps \eqref{eq:spec-maps}:
\begin{equation}\label{eq:comp-eta-psi}
 \eta_j \psi^i = \begin{cases}
  \psi^i\eta_j & \text{if }j\leq i \\
  0 & \text{if }j=i+1 \\
  \psi^i\eta_{j-1} & \text{if }j>i+1
 \end{cases}
 \qquad\qquad
  \eta_j \psi_i = \begin{cases}
  \psi_i\eta_j & \text{if }j< i \\
  \psi_i(\eta_j+\eta_{j+1}) & \text{if }j=i \\
  \psi_i\eta_{j+1} & \text{if }j\geq i+1
 \end{cases}
\end{equation}

There is a natural transformation $\Upsilon:D^{A+1}\oo\bw^{B-1} \lra D^{A}\oo\bw^{B}$ given by a Koszul map, which can be written explicitly when evaluated on $\kk^n$ as
\[\Upsilon(x) = \sum_{i=1}^n \eta_i(x) \wedge e_i,\]
and it has the property that $\coker\Upsilon = \bb{W}_{(A,1^B)}$ (and $\ker\Upsilon = \bb{W}_{(A+1,1^{B-1})}$). We will always take the point of view that $\Upsilon$ gives a presentation of the Weyl module $\bb{W}_{(A,1^B)}$, and when we write $e^{\a} \oo e_{\b}\in \bb{W}_{(A,1^B)}(\kk^n)$, it should be interpreted as the residue class of $e^{\a} \oo e_{\b}\in (D^{A}\oo\bw^{B})(\kk^n)$ modulo the image of $\Upsilon$. The standard pictorial visualization of $e^{\a} \oo e_{\b}\in \bb{W}_{(A,1^B)}(\kk^n)$ is in the form of a Young tableau of shape $(A,1^B)$:
\begin{equation}\label{eq:tableau-from-mon}
\begin{aligned}
\begin{ytableau}
1 & \dots & 1 & 2 & \dots & 2 & \dots & n & \dots & n\\
\b_1 \\
\b_2 \\
\b_3 \\
\vdots \\
\end{ytableau}
\end{aligned}
\end{equation}
where in the first row $1$ is repeated $\a_1$ times, $2$ is repeated $\a_2$ times, etc. The image of $\Upsilon$ gives the straightening relations among such tableaux (see also \cite{mal-ste-hookWeyl}*{Section~2.2}).

We will be interested in a twisted version of $\Upsilon$ which is no longer a natural transformation between functors. It is defined (for all $A,B,n$) as
\begin{equation}\label{eq:def-Phi-divpow-wedge}
 \Phi:\left(D^A\oo\bw^B\right)(\kk^n) \lra \left(D^{A-1}\oo\bw^{B+1}\right)(\kk^{n+1}),\quad\quad \Phi(x) = \sum_{1\leq t\leq s\leq n}(-1)^s\psi^s(\eta_t(x))\wedge e_{s+1}
\end{equation}

\begin{example}\label{ex:Phi-div-wedge}
 If $n=4$, $A=6$, $B=2$, and $x=e_1^{(3)}e_2^{(1)}e_4^{(2)} \oo e_1\wedge e_3$, then $\Phi(x)$ is
 \[
 \begin{aligned}
 & e_1^{(2)}e_2^{(1)}e_4^{(2)} \oo e_1\wedge e_3\wedge e_5 - e_1^{(2)}e_2^{(1)}e_5^{(2)} \oo e_1\wedge e_3\wedge e_4 + e_1^{(2)}e_2^{(1)}e_5^{(2)} \oo e_1\wedge e_4\wedge e_3 - e_1^{(2)}e_3^{(1)}e_5^{(2)} \oo e_1\wedge e_4\wedge e_2   \\
 +\ & e_1^{(3)}e_4^{(2)} \oo e_1\wedge e_3\wedge e_5 - e_1^{(3)}e_5^{(2)} \oo e_1\wedge e_3 \wedge e_4 + e_1^{(3)}e_5^{(2)} \oo e_1\wedge e_4 \wedge e_3 \\
 +\ & e_1^{(3)}e_2^{(1)}e_4^{(1)} \oo e_1\wedge e_3\wedge e_5
 \end{aligned}
 \]
 where the rows correspond to $t=1,2,4$, and in each row the terms are aligned from $s=4$ down to $s=t$. Notice that the weight of $x$ is $(4,1,1,2)$, which has full support, but this is not necessarily the case for the terms in $\Phi(x)$: the first term in the second row has weight $(4,0,1,2,1)$, while the others have weight $(4,0,1,1,2)$.
\end{example}

\begin{lemma}\label{lem:Phi-circ-Upsilon}
 We have that $\Phi\circ\Upsilon + \Upsilon\circ\Phi=0$, and therefore  \eqref{eq:def-Phi-divpow-wedge} induces maps
 \[\Phi:\bb{W}_{(A,1^B)}(\kk^n) \lra \bb{W}_{(A-1,1^{B+1})}(\kk^{n+1}).\]
\end{lemma}

\begin{proof} We consider the diagram
\[
\xymatrix{
\left(D^{A+1}\oo\bw^{B-1}\right)(\kk^n) \ar[r]^{\Phi} \ar[d]^{\Upsilon} & \left(D^A\oo\bw^B\right)(\kk^{n+1}) \ar[d]^{\Upsilon} \\
\left(D^A\oo\bw^B\right)(\kk^n) \ar[r]^{\Phi} & \left(D^{A-1}\oo\bw^{B+1}\right)(\kk^{n+1})
}
\]
and consider any element $x\in\left(D^{A+1}\oo\bw^{B-1}\right)(\kk^n)$. We have
\[
\Upsilon\circ\Phi(x) = \sum_{i=1}^{n+1}\sum_{1\leq t\leq s\leq n}(-1)^s \eta_i(\psi^s(\eta_t(x)))\wedge e_{s+1}\wedge e_i
\]
Using that $\eta_{s+1}\psi^s=0$, we can divide the above sum into terms where $i\leq s$ for which $\eta_i\psi^s=\psi^s\eta_i$, and terms where $i\geq s+2$ for which $\eta_i\psi^s=\psi^s\eta_{i-1}$. We get (after making the substitution $i\mapsto i+1$ in the second sum)
\[\sum_{\substack{1\leq t\leq s\leq n \\ 1\leq i\leq s}}(-1)^s \psi^s(\eta_i(\eta_t(x)))\wedge e_{s+1}\wedge e_i  +\sum_{\substack{1\leq t\leq s\leq n \\ s+1\leq i\leq n}}(-1)^s \psi^s(\eta_i(\eta_t(x)))\wedge e_{s+1}\wedge e_{i+1}\]
Using moreover that $\psi^s(y\wedge e_i)$ equals $\psi^s(y)\wedge e_i$ for $i\leq s$ and $\psi^s(y)\wedge e_{i+1}$ for $i\geq s+1$, as well as the fact that $\eta_i$ and $\eta_t$ commute, we can regroup the terms and further rewrite this as
\[-\sum_{\substack{1\leq t\leq s\leq n \\ 1\leq i\leq n}}(-1)^s \psi^s(\eta_t(\eta_i(x)\wedge e_i))\wedge e_{s+1} = - \Phi\circ\Upsilon(x) \]
which proves the desired relation. The conclusion about the induced maps follows from the fact that $\Upsilon$ gives the presentation map for the respective Weyl modules.
\end{proof}

\begin{lemma}\label{lem:aux-form-Ups}
 We have for $x\in \left(D^A\oo\bw^B\right)(\kk^n)$
 \[\Upsilon(x) = \sum_{\substack{1\leq t\leq s\leq i\leq n \\ i\leq s+1}} (-1)^{i+s} \psi_i(\psi^s(\eta_t(x)))\wedge\psi_i(e_{s+1})\]
\end{lemma}

\begin{proof} Notice that in the above sum we only allow $i=s$ or $i=s+1$, and the latter occurs if and only if $s<n$. Using \eqref{eq:cosimplicial-psi} we get $\psi_i\psi^s=\text{identity}$, and moreover we have $\psi_i(e_{s+1})=e_s$ if $i=s$ and $\psi_i(e_{s+1})=e_{s+1}$ if $i=s+1$. The sum then simplifies to
\[ \sum_{1\leq t\leq s\leq n} \eta_t(x)\wedge(e_s-e_{s+1}) \]
where we make the convention that $e_{n+1}=0$. Since $\sum_{s=t}^n (e_s-e_{s+1}) = e_t$, we can simplify the expression further to
\[\sum_{1\leq t\leq n} \eta_t(x)\wedge e_t = \Upsilon(x),\]
which gives the desired conclusion.
\end{proof}

\begin{lemma}\label{lem:Ups=Phipar+parPhi}
 We have that $\Phi\circ\partial+\partial\circ\Phi + \Upsilon = 0$.
\end{lemma}

\begin{proof}
We consider the diagram
\[
\xymatrix{
\left(D^{A}\oo\bw^{B}\right)(\kk^n) \ar[r]^{\Phi} \ar[d]^{\partial}  \ar[dr]^{\Upsilon} & \left(D^{A-1}\oo\bw^{B+1}\right)(\kk^{n+1}) \ar[d]^{\partial} \\
\left(D^A\oo\bw^B\right)(\kk^{n-1}) \ar[r]^{\Phi} & \left(D^{A-1}\oo\bw^{B+1}\right)(\kk^{n})
}
\]
and consider any element $x\in\left(D^{A}\oo\bw^{B}\right)(\kk^n)$. We have
\begin{equation}\label{eq:Phi-par-x}
 \Phi(\partial(x)) = \sum_{1\leq t\leq s\leq n-1}\sum_{i=1}^{n-1}(-1)^{i+s-1}\psi^s(\eta_t(\psi_i(x)))\wedge e_{s+1}.
\end{equation}
We can swap $\eta_\bullet$ with $\psi_\bullet$ using \eqref{eq:comp-eta-psi}, and divide \eqref{eq:Phi-par-x} into two parts $L_1$, $R_1$, where
\[ L_1 = \sum_{\substack{1\leq t\leq s\leq n-1 \\ 1\leq i\leq t}}(-1)^{i+s-1}\psi^s(\psi_i(\eta_{t+1}(x)))\wedge e_{s+1},\qquad 
  R_1 = \sum_{\substack{1\leq t\leq s\leq n-1 \\ t\leq i\leq n-1}}(-1)^{i+s-1}\psi^s(\psi_i(\eta_{t}(x)))\wedge e_{s+1}.
\]
noting that the terms in \eqref{eq:Phi-par-x} that involve $\eta_t\psi_t$ (where $i=t$) contribute $\psi_t\eta_{t+1}$ to the left sum, and $\psi_t\eta_{t}$ to the right sum. We make the substitution $t\mapsto t-1$ in $L_1$, and then for $1\leq t\leq s\leq n-1$ we regroup the terms in $L_1+R_1$ by allowing $1\leq i\leq n-1$. We are left with additional terms in $L_1$ corresponding to $t=s+1$, hence:
\[ L_1+R_1 = \sum_{\substack{1\leq t\leq s\leq n-1 \\ 1\leq i\leq n-1}}(-1)^{i+s-1}\psi^s(\psi_i(\eta_t(x)))\wedge e_{s+1} + \sum_{1\leq i\leq s\leq n-1}(-1)^{i+s-1}\psi^s(\psi_i(\eta_{s+1}(x)))\wedge e_{s+1}.\]
Using Lemma~\ref{lem:aux-form-Ups}, we can cancel out $\Upsilon(x)$ with the terms where $i=s,s+1$ in the expression below
\[
 \Upsilon(x) + \partial(\Phi(x)) = \Upsilon(x) + \sum_{i=1}^{n}\sum_{1\leq t\leq s\leq n-1}(-1)^{i+s-1}\psi_i(\psi^s(\eta_t(x)))\wedge \psi_i(e_{s+1})
\]
We can further swap $\psi_\bullet$ with $\psi^\bullet$ using \eqref{eq:cosimplicial-psi}, and divide the result into two parts $L_2$, $R_2$, where
\[ L_2 = \sum_{\substack{1\leq t\leq s\leq n \\ 1\leq i\leq s-1}}(-1)^{i+s-1}\psi^{s-1}(\psi_i(\eta_t(x)))\wedge e_s,\qquad 
  R_2 = \sum_{\substack{1\leq t\leq s\leq n \\ s+2\leq i\leq n}}(-1)^{i+s-1}\psi^s(\psi_{i-1}(\eta_{t}(x)))\wedge e_{s+1}.
\]
We make the substitutions $s\mapsto s+1$ in $L_2$ and $i\mapsto i+1$ in $R_2$, and then for $1\leq t\leq s\leq n-1$ we regroup the terms in $L_2+R_2$ by allowing $1\leq i\leq n-1$. We are again left with additional terms in $L_2$ corresponding to $t=s+1$:
\[ L_2+R_2 = \sum_{\substack{1\leq t\leq s\leq n-1 \\ 1\leq i\leq n-1}}(-1)^{i+s}\psi^s(\psi_i(\eta_t(x)))\wedge e_{s+1} + \sum_{1\leq i\leq s\leq n-1}(-1)^{i+s}\psi^s(\psi_i(\eta_{s+1}(x)))\wedge e_{s+1}.\]
Notice that the formula for $L_2+R_2$ differs from $L_1+R_1$ by a factor of $(-1)$, proving the desired relation.
\end{proof}

\begin{example}\label{ex:Phidel+delPhi}
 Consider $n=3$, $A=4$, $B=2$, and $x=e_1^{(3)}e_2^{(1)} \oo e_1\wedge e_3$. Similarly to Example~\ref{ex:Phi-div-wedge} we get
 \[ \Phi(x) = -2e_1^{(2)}e_2^{(1)} \oo e_1\wedge e_3\wedge e_4 + e_1^{(2)}e_3^{(1)} \oo e_1\wedge e_2\wedge e_4 - 2e_1^{(3)} \oo e_1\wedge e_3\wedge e_4,\]
 and therefore
 \[ \partial(\Phi(x)) = \left(-6e_1^{(3)} + e_1^{(2)}e_2^{(1)} + e_1^{(2)}e_3^{(1)}\right) \oo e_1\wedge e_2\wedge e_3.\]
 On the other hand we have 
 \[ \partial(x) = \left(4e_1^{(4)} - e_1^{(3)}e_2^{(1)}\right) \oo e_1\wedge e_2,\]
 and one can check that after applying $\Phi$ we get
 \[ \Phi(\partial(x)) = \left(7e_1^{(3)} - e_1^{(2)}e_2^{(1)} - e_1^{(2)}e_3^{(1)}\right) \oo e_1\wedge e_2\wedge e_3.\]
 This shows that $(\Phi\circ\partial+\partial\circ\Phi)(x) = e_1^{(3)} \oo e_1\wedge e_2\wedge e_3 = -\Upsilon(x)$.
\end{example}

As seen in Examples~\ref{ex:Phi-div-wedge},~\ref{ex:Phidel+delPhi}, $\Phi$ does not restrict to a map between weight spaces of full support. Nevertheless, the next result shows that $\Phi$ is compatible with the filtration by the extended complexes $\hat{\mc{F}}^a_{\bullet}$.

\begin{proposition}\label{prop:Phi-map-Fhat}
 The operator $\Phi$ induces morphisms of complexes
 \[ \Phi:  \hat{\mc{F}}^a_{\bullet}\left(\bb{W}_{(A,1^B)}\right) \lra \hat{\mc{F}}^{a-1}_{\bullet+1}\left(\bb{W}_{(A-1,1^{B+1})}\right)  \]
 if we make the convention that $\pd_{\bullet+1} = -\pd_{\bullet}$.
\end{proposition}

\begin{proof}
 We write $\mc{P}=\bb{W}_{(A,1^B)}$ and $\mc{P}'=\bb{W}_{(A-1,1^{B+1})}$ and recall from Lemma~\ref{lem:Phi-circ-Upsilon} that $\Phi$ induces a well-defined map $\mc{P}(\kk^n)\lra\mc{P}'(\kk^{n+1})$. If $\epsilon_1=(1,0,\cdots)$, $\epsilon_2,\cdots$ are the standard unit vectors in $\bb{Z}^n$ (and $\bb{Z}^{n+1}$), then $\Phi$ takes $\mc{P}_{\ul{d}}$ to a sum of weight spaces $\mc{P}'_{\ul{d}'}$, where each $\ul{d}'$ has the form
 \[\ul{d}' = \psi^s(\ul{d}) - \epsilon_t + \epsilon_{s+1}\quad\text{ for some }1\leq t\leq s\leq n.\]
 Notice that if $d_1\geq a$ then $d'_1\geq d_1-1 \geq a-1$. Also if $d_n\neq 0$ then $d'_{n+1}\neq 0$: if $s=n$ then $d'_{n+1}=1$, while for $1\leq s\leq n-1$ we have $d'_{n+1}=d_n$. This shows that $\Phi$ defines maps $\hat{\mc{F}}^a_{n}(\mc{P})$ to $\hat{\mc{F}}^{a-1}_{n+1}(\mc{P}')$ for all $n$, so it remains to verify that $\Phi$ commutes with the differentials in $\hat{\mc{F}}^a_{\bullet}(\mc{P})$ and $\hat{\mc{F}}^{a-1}_{\bullet+1}(\mc{P}')$, which follows from Lemma~\ref{lem:Ups=Phipar+parPhi}.
\end{proof}

Combining Proposition~\ref{prop:Phi-map-Fhat} with Lemma~\ref{lem:Fdegenerate} concludes the construction of the morphism of complexes \eqref{eq:Phi-on-Fa}.

\begin{corollary}\label{cor:Phi-on-Fa}
 The operator $\Phi$ induces morphisms of complexes
  \[ \Phi:  \mc{F}^a_{\bullet}\left(\bb{W}_{(A,1^B)}\right) \lra \mc{F}^{a-1}_{\bullet+1}\left(\bb{W}_{(A-1,1^{B+1})}\right)  \]
obtained as the composition 
\[\mc{F}^a_{\bullet}\left(\bb{W}_{(A,1^B)}\right) \hookrightarrow  \hat{\mc{F}}^a_{\bullet}\left(\bb{W}_{(A,1^B)}\right) \overset{\Phi}{\lra} \hat{\mc{F}}^{a-1}_{\bullet+1}\left(\bb{W}_{(A-1,1^{B+1})}\right) \onto \mc{F}^{a-1}_{\bullet+1}\left(\bb{W}_{(A-1,1^{B+1})}\right).\]
\end{corollary}

\begin{example}\label{ex:Phi-on-Fa}
One can think of the evaluation of $\Phi(x)$ in Corollary~\ref{cor:Phi-on-Fa} as applying the formula \eqref{eq:def-Phi-divpow-wedge}, and then ``ignoring" all the summands that do not have full support. With the notation in Example~\ref{ex:Phidel+delPhi}, we take $a=4$ and interpret $x$ as an element of $\mc{F}^a_n(\bb{W}_{(A,1^B)})=\mc{F}^4_3(\bb{W}_{(4,1^2)})$. We get
\[ 
\Yvcentermath1
\Phi\left(\young(1112,1,3)\right) = (-2)\cdot\young(112,1,3,4) + \young(113,1,2,4)
\]
and we note that the tableau corresponding to $2e_1^{(3)} \oo e_1\wedge e_3\wedge e_4$ was ignored, since the corresponding weight $(4,0,1,1)$ did not have full support. Using the straightening relations we get
\[ 
\Yvcentermath1
\young(1112,1,3) = -\young(1111,2,3),\quad \young(112,1,3,4)=-\young(111,2,3,4),\quad \young(113,1,2,4)=-\young(111,3,2,4)=\young(111,2,3,4),
\]
hence the above equality can be rewritten more simply as
\[
\Yvcentermath1
\Phi\left(\young(1111,2,3)\right) = 3 \cdot \young(111,2,3,4)
\]
The same conclusion arises from applying the formula \eqref{eq:def-Phi-divpow-wedge}:
\[ \Phi(e_1^{(4)}\oo e_2\wedge e_3) =  e_1^{(3)}\oo e_2\wedge e_3 \wedge e_4 - e_1^{(3)}\oo e_2\wedge e_4 \wedge e_3 + e_1^{(3)}\oo e_3\wedge e_4 \wedge e_2 = 3\cdot e_1^{(3)}\oo e_2\wedge e_3 \wedge e_4,\]
but notice that this time all elements in the image have full support, hence none are ignored.
\end{example}

\subsection{Divided powers of $\Phi$}
\label{subsec:div-pow-Phi}

The next goal is to define the maps \eqref{eq:def-PhiB-pre}, verify that they satisfy \eqref{eq:PhiB-div-power}, and deduce that $\Phi^{[B]}$ are morphisms of complexes. We write $[N]=\{1,\cdots,N\}$, consider ordered set partitions of $[N]$, denoted $I_{\bullet}=(I_1,\cdots,I_n)\vdash[N]$ and given by
\[ [N] = I_1 \sqcup \cdots \sqcup I_n\text{ with }I_k\neq\emptyset,\text{ and we let }i_k = \min(I_k)\text{ for }k=1,\cdots,n.\]
For a tuple $\ul{d}=(d_1,\cdots,d_n)\in\bb{Z}^n_{>0}$ we consider
\[ \mf{Par}(\ul{d};N) = \{ (I_1,\cdots,I_n)\vdash[N] : i_1\leq \cdots \leq i_n\text{ and } |I_k|\leq d_k\text{ for all }k=1,\cdots,n\}.\]
Given a tuple $\ul{d}$ and a partition $I_{\bullet}\in\mf{Par}(\ul{d};N)$, we define $\a=\a(\ul{d};I_{\bullet})$, $\b=\b(\ul{d};I_{\bullet})$ via
\begin{equation}\label{eq:defa-defb}
 \a = (d_1+1-|I_1|,\cdots,d_n+1-|I_n|) \in \bb{Z}^n_{>0},\qquad \beta=\{\beta_1<\cdots<\beta_{N-n}\} = [N]\setminus\{i_1,\cdots,i_n\},
\end{equation}
and let
\begin{equation}\label{eq:def-sgnI}
 \sgn(\ul{d};I_{\bullet}) = (-1)^{(\b_1-1)+\cdots+(\b_{N-n}-1)},\quad m(\ul{d};I_{\bullet}) = e_{i_1}^{(\a_1)}\cdots e_{i_n}^{(\a_n)} \oo e_{\b}.
\end{equation}
Using this notation, we let
\begin{equation}\label{eq:def-PhiB}
 \Phi^{[B]}(e^{\ul{d}}) = \sum_{I_{\bullet}\in\mf{Par}(\ul{d};n+B)} \sgn(\ul{d};I_{\bullet}) \cdot m(\ul{d};I_{\bullet})
\end{equation}
and proceed to verifying the identity \eqref{eq:PhiB-div-power} after some examples.

\begin{example}\label{ex:B=1}
 If $B=1$, every partition $I_{\bullet}\in \mf{Par}(\ul{d};n+1)$ contains a unique set $I_k$ of size $2$ (where $d_k\geq 2$), hence it is uniquely characterized by a pair $(k,s)$, where
 \[ 1\leq k\leq s\leq n,\text{ and }I_k=\{k,s+1\},\ d_k\geq 2.\]
 We have $I_j = \{\psi^s(j)\}$ for $j\neq k$. For each such $I_{\bullet}$ we get
 \[ \a(\ul{d};I_{\bullet})=(d_1,\cdots,d_k-1,\cdots,d_n),\quad\beta(\ul{d};I_{\bullet})=\{s+1\},\quad \sgn(\ul{d};I_{\bullet})=(-1)^s,\]
 and therefore we can identify
 \[ m(\ul{d};I_{\bullet}) = e_1^{(d_1)}\cdots e_k^{(d_k-1)}\cdots e_{s}^{(d_s)} e_{s+2}^{(d_{s+1})}\cdots e_{n+1}^{(d_n)} \oo e_{s+1} = \psi^s(\eta_k(e^{\ul{d}})) \oo e_{s+1}.\]
 It follows that 
 \[ \Phi^{[1]}(e^{\ul{d}}) = \sum_{\substack{1\leq k\leq s\leq n \\ d_k\geq 2}} \psi^s(\eta_k(e^{\ul{d}})) \oo e_{s+1} = \sum_{1\leq t\leq s\leq n}(-1)^s\psi^s(\eta_t(e^{\ul{d}})) \oo e_{s+1} = \Phi(e^{\ul{d}}),\]
 where the potentially larger sum on the right differs by terms where $t=k$ and $d_k=1$, for which the corresponding term $\psi^s(\eta_k(e^{\ul{d}})) \oo e_{s+1}$ has weight $\ul{d}'\in\bb{Z}^{n+1}$ without full support (namely $d'_k=0$), and can therefore be ignored as explained in Example~\ref{ex:Phi-on-Fa}.
\end{example}

\begin{example}\label{ex:PhiB-concrete}
 For a more concrete example, we take $n=2$, $\ul{d}=(3,1)$, $B=2$, and compute as explained in the previous section 
 \[ 
 \Yvcentermath1
 \Phi(\young(1113)) = \young(112,3)-\young(113,2),\quad\text{ and } (\Phi\circ\Phi)(\young(1113)) = -2\cdot\young(12,3,4)+2\cdot\young(13,2,4)-2\cdot\young(14,2,3).
 \]
The new invariants we are considering in this section are recorded below.
\[ \mf{Par}(\ul{d};n+B) = \{(\{1,3,4\},\{2\}), (\{1,2,4\},\{3\}), (\{1,2,3\},\{4\})\},\]
 \begin{center}
\renewcommand{\arraystretch}{2} 
\begin{tabular}{>{\centering\arraybackslash}m{2.5cm}|>{\centering\arraybackslash}m{2.5cm}|>{\centering\arraybackslash}m{2.5cm}|>{\centering\arraybackslash}m{2.5cm}|>{\centering\arraybackslash}m{2.5cm}}
\textbf{$I_{\bullet}$} & \textbf{$\a(\ul{d};I_{\bullet})$} & \textbf{$\b(\ul{d};I_{\bullet})$} & \textbf{$\sgn(\ul{d};I_{\bullet})$} & \textbf{$m(\ul{d};I_{\bullet})$} \\
\hline
\{1,3,4\},\{2\} & (1,1) & \{3,4\} & -1 & $e_1^{(1)}e_2^{(1)} \oo e_3\wedge e_4$ \\
\hline
\{1,2,4\},\{3\} & (1,1) & \{2,4\} & 1 & $e_1^{(1)}e_3^{(1)} \oo e_2\wedge e_4$ \\
\hline
\{1,2,3\},\{4\} & (1,1) & \{2,3\} & -1 & $e_1^{(1)}e_4^{(1)} \oo e_2\wedge e_3$ \\
\end{tabular}
\end{center}
Based on \eqref{eq:def-PhiB} we obtain
 \[ 
 \Yvcentermath1
 \Phi^{[2]}(\young(1113)) = -\young(12,3,4)+\young(13,2,4)-\young(14,2,3),
 \]
 hence $\Phi\circ\Phi = 2\cdot\Phi^{[2]}$ as intended.
\end{example}

For a pair $(k,s)$ satisfying $|I_k| < d_k$ and $i_k\leq s\leq N$, we define
\begin{equation}\label{eq:def-Sigma-js-I}
 \Sigma_{k,s}(I_1,\cdots,I_n) = (\psi^s(I_1),\cdots,\psi^s(I_k) \cup \{s+1\},\cdots,\psi^s(I_n)) \in \mf{Par}(\ul{d};N+1) .
\end{equation}

\begin{lemma}\label{lem:multipl-SigmaI}
 Each $J_{\bullet}=(J_1,\cdots,J_n)\in \mf{Par}(\ul{d};N+1)$ can be expressed in exactly $(N+1-n)$ ways as $\Sigma_{k,s}(I_\bullet)$ for some $k,s$ and $I_{\bullet}\in \mf{Par}(\ul{d};N)$ as above.
\end{lemma}

\begin{proof} We write $j_k = \min(J_k)$ for $k=1,\cdots,n$, and assume that $J_{\bullet} = \Sigma_{k,s}(I_{\bullet})$. It is clear that $i_k = \psi^s(i_k) = j_k$, since $s\geq i_k$ and $\psi^s$ is non-decreasing. Moreover, $(s+1)$ is a non-minimal element in a unique $J_k$. In particular $J_{\bullet}$ and $s$ determines $k$, and we can recover $I_{\bullet}$ as
\begin{equation}\label{eq:IfromJ}
 I_{\bullet} = \left(\psi_s(J_1),\cdots,\psi_s(J_k\setminus\{s+1\}),\cdots,\psi_s(J_n)\right).
\end{equation}
Every element of $[N+1]$ which is not minimal in any of the sets $J_k$ can be expressed uniquely as $s+1$ for $1\leq s\leq N$. Each such $s$ determines a unique index $k$ such that $(s+1)\in J_k$ and a unique partition $I_{\bullet}$ as in \eqref{eq:IfromJ}. Since there are $(N+1-n)$ elements in $[N+1]$ which are not minimal in any of $J_k$, the conclusion follows.
\end{proof}

\begin{example}\label{ex:Sig-ks-I}
 Suppose that $n=3$ and $\ul{d}$ satisfies $d_1\geq 1$, $d_2\geq 3$, $d_3\geq 2$. If we let
 \[ I_{\bullet} = (\{1\},\{2,3\},\{4,5\}),\quad k=2,\quad s=4,\text{ then we get}\quad \Sigma_{k,s}(I_{\bullet}) = (\{1\},\{2,3,5\},\{4,6\}) =: J_{\bullet}\]
 With the notation in Lemma~\ref{lem:multipl-SigmaI} we have $N=5$, so there should be $(N+1-n)=3$ ways of expressing $J_{\bullet}$ as $\Sigma_{k,s}(I_{\bullet})$, each corresponding to a choice of $(s+1)\in\{3,5,6\}$ as a non-minimal element in the sets $J_{\bullet}$, and of $k$ as the unique index for which $(s+1)\in J_k$. We have using \eqref{eq:IfromJ}
 \[
 \begin{aligned}
  s+1=3&:\ k=2\quad\text{ and }\quad I_{\bullet} = (\psi_2(J_1),\psi_2(J_2\setminus\{3\}),\psi_2(J_3)) = (\{1\},\{2,4\},\{3,5\}),\\
  s+1=5&:\ k=2\quad\text{ and }\quad I_{\bullet} = (\psi_4(J_1),\psi_4(J_2\setminus\{3\}),\psi_4(J_3)) = (\{1\},\{2,3\},\{4,5\}),\text{ as seen above},\\
  s+1=6&:\ k=3\quad\text{ and }\quad I_{\bullet} = (\psi_5(J_1),\psi_5(J_2),\psi_5(J_3\setminus\{3\})) = (\{1\},\{2,3,5\},\{4\}).\\
 \end{aligned}
 \]
\end{example}

\begin{proof}[Proof of \eqref{eq:PhiB-div-power}]
 We have
 \[\Phi\circ\Phi^{[B]}(e^{\ul{d}}) = \sum_{\substack{1\leq t\leq s\leq n \\ I_{\bullet}\in\mf{Par}(\ul{d};n+B)}} (-1)^s\sgn(\ul{d};I_{\bullet}) \psi^s(\eta_t(m(\ul{d};I_{\bullet})))\wedge e_{s+1}\]
 Note that $\eta_t(m(\ul{d};I_{\bullet}))=0$ if $t\neq i_k$, and $\psi^s(\eta_t(m(\ul{d};I_{\bullet})))\wedge e_{s+1}$ does not have full support if $t=i_k$ and $|I_k|=d_k$ (see Example~\ref{ex:B=1}). We may then restrict the summation above to terms where $t=i_k$ (for some $k$) and $d_k>|I_k|$.
 
 For such terms we let $J_{\bullet}=\Sigma_{k,s}(I_{\bullet})$, and note that
 \[ \psi^s(\eta_t(m(\ul{d};I_{\bullet})))\wedge e_{s+1} = (-1)^r \cdot m(\ul{d};J_{\bullet}),\]
where $r$ denotes the number of elements of $\b(\ul{d};I_{\bullet})$ which are $\geq s+1$, and therefore can be computed as
\[ r = \sum_{x\in\b(\ul{d};I_{\bullet})} (\psi^s(x)-x).\]
Using the fact that $\b(\ul{d};J_{\bullet}) = \{s+1\} \cup \psi^s(\b(\ul{d};I_{\bullet}))$, it follows that
\[\sum_{x\in\b(\ul{d};J_{\bullet})}(x-1) - \sum_{x\in\b(\ul{d};I_{\bullet})}(x-1) = (s+1-1) + \sum_{x\in\b(\ul{d};I_{\bullet})} (\psi^s(x)-x) = s+r.\]
This shows that
\[ (-1)^s\sgn(\ul{d};I_{\bullet}) = \sgn(\ul{d};J_{\bullet})(-1)^r,\]
hence
\[ (-1)^s\sgn(\ul{d};I_{\bullet}) \psi^s(\eta_t(m(\ul{d};I_{\bullet})))\wedge e_{s+1} = \sgn(\ul{d};J_{\bullet})m(\ul{d};J_{\bullet}).\]
Using Lemma~\ref{lem:multipl-SigmaI}, each of the terms $\sgn(\ul{d};J_{\bullet})m(\ul{d};J_{\bullet})$, with $J_{\bullet}\in\mf{Par}(\ul{d};n+B+1)$ can be expressed in precisely $(B+1)$ ways as $\Sigma_{k,s}(I_\bullet)$, which yields the desired identity $\left(\Phi\circ\Phi^{[B]}\right)(e^{\ul{d}}) = (B+1)\Phi^{[B+1]}(e^{\ul{d}})$.
\end{proof}

Finally, we prove by induction on $B$ that $\Phi^{[B]}$ is a morphism of complexes: when $B=1$ this follows from Example~\ref{ex:B=1} and Corollary~\ref{cor:Phi-on-Fa}. As explained in our outline, our constructions are compatible with base change, so it suffices to consider the case when $\kk=\bb{Z}$. Using induction and the fact that $\Phi$ is a morphism of complexes, it follows from \eqref{eq:PhiB-div-power} that $\Phi\circ\Phi^{[B]} = (B+1)\cdot \Phi^{[B+1]}$ is also a morphism of complexes. Since $(B+1)$ is a non-zero divisor in $\kk=\bb{Z}$, it follows that $\Phi^{[B+1]}$ is a morphism of complexes, proving the induction step.

\subsection{The proof that $\Phi^{[B]}$ is a quasi-isomorphism}
\label{subsec:PhiB-quasi-isom}

In the previous section we have constructed the morphisms of complexes \eqref{eq:def-PhiB-pre}, and our final goal is to prove that these are quasi-isomorphisms for each $a\geq 1$. We prove this by descending induction on $a$, noting that both complexes are $0$ when $a>A$. It then suffices to show that $\Phi^{[B]}$ induces isomorphisms on the associated graded quotients
\begin{equation}\label{eq:PhiB-gr-gen}
  \Phi^{[B]}:\mf{gr}\mc{F}^{a+B}_{\bullet}(\bb{W}_{(d)}) \lra \mf{gr}\mc{F}^a_{\bullet+B}(\bb{W}_{(A,1^B)})
\end{equation}
To describe the map explicitly we note that by \eqref{eq:grFan-weights} the input $e^{\ul{d}}$ in \eqref{eq:def-PhiB} has the form $\ul{d}=(a+B,d_2,\cdots,d_n)$ with $d_2+\cdots+d_n=d-(a+B)=\Delta$, and the non-zero terms in the output have weights of the form $\ul{d'}=(a,\cdots)=(d_1+1-|I_1|,\cdots)$. It follows that $|I_1|=B+1$, which forces $|I_k|=1$ for $k\geq 2$. In particular the partitions $I_{\bullet}$ are uniquely determined by
\[ I_1 = \{1\} \cup \b,\quad\text{ where the set $\b=\b(\ul{d};I_{\bullet})$ has the form }\b = \{2\leq \b_1<\cdots<\b_B\leq n+B\}.\]
Writing $\sgn(\ul{d};\b)$ and $m(\ul{d};\b)$ in \eqref{eq:def-PhiB} we get
\begin{equation}\label{eq:PhiB-ed-on-grF}
\begin{aligned}
& \Phi^{[B]}(e^{\ul{d}}) = \sum_{2\leq\b_1<\cdots<\b_B\leq n+B} \sgn(\ul{d};\b)\cdot m(\ul{d};\b),\\
\text{ where }&\sgn(\ul{d};\b) = (-1)^{\sum_{i=1}^B (\b_i-1)},\text{ and } m(\ul{d};\b) = e_1^{(a)}e_{i_2}^{(d_2)}\cdots e_{i_n}^{(d_n)} \oo e_\b.\\
\end{aligned}
\end{equation}
We will refer to a monomial $m=e^{\a}\oo e_{\b}\in\mf{gr}\mc{F}^a_{n+B}(\bb{W}_{(A,1^B)})$ as \defi{standard} if the corresponding tableau \eqref{eq:tableau-from-mon} is standard (weakly increasing in the first row, and strictly increasing in the first column), and recall that standard monomials form a basis of $\mf{gr}\mc{F}^a_{n+B}(\bb{W}_{(A,1^B)})$. We say that $m$ is \defi{terminal} if $\b_i=n+i$ for all $1\leq i\leq B$, and $\a_i=0$ for $i>n$. This is equivalent to the condition that
\[ m = m(\ul{d};\b)\text{ for }\b=\{n+1,\cdots,n+B\},\ \ul{d} = (a+B,d_2,\cdots,d_n)\text{ with full support}.\]
Note that all the monomials in \eqref{eq:PhiB-ed-on-grF} are standard since $a\geq 1$, and precisely one of them is terminal.

\begin{lemma}\label{lem:XiB-retract-PhiB}
 The map \eqref{eq:PhiB-gr-gen} admits a retraction $\Theta_{[B]}$ defined on standard monomials $m\in\mf{gr}\mc{F}^a_{n+B}(\bb{W}_{(A,1^B)})$ by
\[ 
\Theta_{[B]}(m) = \begin{cases}
 \sgn(\ul{d};\b)\cdot e^{\ul{d}} & \text{if }m=m(\ul{d};\b)\text{ is terminal}, \\
 0 & \text{otherwise}.
\end{cases}
\]
\end{lemma}

\begin{proof} It is clear from \eqref{eq:PhiB-ed-on-grF} that $\Theta_{[B]}$ provides a retract for $\Phi^{[B]}$, provided we verify that $\Theta_{[B]}$ is a morphism of complexes. Similar to Proposition~\ref{prop:Phi-map-Fhat}, our convention is that the differential in $\mf{gr}\mc{F}^a_{\bullet+B}(\bb{W}_{(A,1^B)})$ is given by $\partial_{\bullet+B}=(-1)^B\partial_{\bullet}$. We consider a standard monomial basis element
\[ m = e^{\a}\oo e_{\b}\in\mf{gr}\mc{F}^a_{n+B}(\bb{W}_{(A,1^B)}),\quad\a=(a,\a_2,\cdots,\a_{n+B}),\quad \b = \{2\leq \b_1<\cdots<\b_B\leq n+B\},\]
and denote the weight of $m$ by $\ul{d}'\in\bb{Z}^{n+B}_{>0}$. We need to verify the identity
\begin{equation}\label{eq:XiB-morph-cxs}
 \Theta_{[B]}\left(\sum_{i=1}^{n+B-1} (-1)^{i+B-1}\psi_i(m)\right) = \sum_{i=1}^{n-1} (-1)^{i-1}\psi_i\left(\Theta_{[B]}(m)\right).
\end{equation}

If $m$ is not terminal, then we claim that $\a_j>0$ for some $j>n$: if this were not the case, then the condition $d'_j>0$ for $j>n$ would force each of $j=n+1,\cdots,n+B$ to be in the set $\b$, hence $\b=\{n+1,\cdots,n+B\}$. Letting $d_1=a+B$ and $d_i=\a_i$ for $i=2,\cdots,n$, it would follow that $m=m(\ul{d};\b)$, a contradiction. Fix now an index $j>n$ with $\a_j>0$. Every specialization $m' = \psi_i(m)$ has the form $e^{\a'}\oo e_{\b'}$ where either $\a'_j>0$ or $\a'_{j-1}>0$. Since $j-1>n-1$, it follows that $m'\in \mf{gr}\mc{F}^a_{n-1+B}(\bb{W}_{(A,1^B)})$ is not terminal. It then follows that both sides of \eqref{eq:XiB-morph-cxs} vanish, so the equality holds.

Assume now that $m=m(\ul{d};\b)$ is terminal, so that $\a=(a,d_2,\cdots,d_n,0,\cdots,0)$, $\b=\{n+1,\cdots,n+B\}$. We have $\Theta_{[B]}(m)=\sgn(\ul{d};\b)\cdot e^{\ul{d}}$, and $\psi_1(e^{\ul{d}}) = e_1^{(d_1+d_2)}\cdots$, where $d_1+d_2=a+B+d_2>a+B$, hence $\psi_1(e^{\ul{d}}) = 0$ in $\mf{gr}\mc{F}^{a+B}_{n-1}(\bb{W}_{(d)})$. It follows that the right side of \eqref{eq:XiB-morph-cxs} is given by
\[ \sgn(\ul{d};\b)\cdot\sum_{i=2}^{n-1} (-1)^{i-1}\psi_i\left(e^{\ul{d}}\right).\]
The same argument shows that $\psi_1(m)=0$ in $\mf{gr}\mc{F}^a_{n-1+B}(\bb{W}_{(A,1^B)})$, and moreover we have $\psi_j(m)=0$ for $j=n+1,\cdots,n+B-1$, because $\psi_j(e_j\wedge e_{j+1}) = e_j\wedge e_j = 0$. Since $\psi_n(m) = e^{\a} \oo \psi_n(e_{\b})$ and $\a_n=d_n\neq 0$, it follows that $\psi_n(m)$ is not terminal, and hence $\Theta_{[B]}(\psi_n(m))=0$. We can then rewrite the left side of \eqref{eq:XiB-morph-cxs} as
\[  \Theta_{[B]}\left(\sum_{i=2}^{n-1} (-1)^{i+B-1}\psi_i(m)\right).\]
If we let $\b'=\{n,\cdots,n+B-1\}$ then for $2\leq i\leq n-1$ we have $\b'=\psi_i(\b)$ and 
\[ \psi_i(m) = {d_i+d_{i+1}\choose d_i} \cdot m(\psi_i(\ul{d});\b'),\quad\text{hence}\quad\Theta_{[B]}\left(\psi_i(m)\right) = {d_i+d_{i+1}\choose d_i}\cdot e^{\psi_i(\ul{d})} = \psi_i\left(e^{\ul{d}}\right). \]
The equality \eqref{eq:XiB-morph-cxs} now follows from the fact that $\sgn(\ul{d};\b)=(-1)^B\cdot\sgn(\ul{d};\b')$, concluding our proof.
\end{proof}

It follows from Lemma~\ref{lem:XiB-retract-PhiB} that the map $H_n(\Theta_{[B]})$ induced by $\Theta_{[B]}$ is a left inverse to $H_n(\Phi^{[B]})$ and in particular it is surjective. By Theorem~\ref{thm:Ext-from-spec-complexes}(2), we have
\[ H_n(\mf{gr}\mc{F}^{a+B}_{\bullet}(\bb{W}_{(d)})) = \Ext^{\Delta+1-n}\left(\bw^\Delta,D^{\Delta}\right),\]
\[ H_n(\mf{gr}\mc{F}^{a}_{\bullet+B}(\bb{W}_{(A,1^B)})) = H_{n+B}(\mf{gr}\mc{F}^{a}_{\bullet}(\bb{W}_{(A,1^B)})) = \Ext^{\Delta+1-n}\left(\bw^b,D^{\Delta}\oo\bw^B\right) = \Ext^{\Delta+1-n}\left(\bw^\Delta,D^{\Delta}\right),\]
where the last equality follows from \cite{kulkarni}*{Theorem~1}. This shows that the source and target of $H_n(\Theta_{[B]})$ are abstractly isomorphic, and since they are finitely generated over $\kk$, the surjection $H_n(\Theta_{[B]})$ must be an isomorphism. We get that $H_n(\Phi^{[B]})$ is an isomorphism as well, hence \eqref{eq:def-PhiB-pre} is a quasi-isomorphism.

\begin{proof}[Proof of Theorem~\ref{thm:Ext-hook=Ext-divided}]
If we let $a=A-\Delta$, consider the partition $\mu=(d-\Delta,1^{\Delta}) = (a+B,1^{d-a-B})$, and take $\mc{P}=\bb{W}_{(d)}$ in Theorem~\ref{thm:Ext-from-spec-complexes}(1), then it follows that
\[ \Ext^i\left(\bb{W}_{(d-\Delta,1^{\Delta})},\bb{W}_{(d)} \right) = H_{\Delta+1-i}\left(\mc{F}^{a+B}_{\bullet}(\bb{W}_{(d)})\right).\]
If instead we take $\mu=(A-\Delta,1^{B+\Delta}) = (a,1^{d-a})$ and $\mc{P}=\bb{W}_{(A,1^B)}$ in Theorem~\ref{thm:Ext-from-spec-complexes}(1), then we get
\[ \Ext^i\left(\bb{W}_{(A-\Delta,1^{B+\Delta})},\bb{W}_{(A,1^B)} \right) = H_{B+\Delta+1-i}\left(\mc{F}^{a}_{\bullet}(\bb{W}_{(A,1^B)})\right) = H_{\Delta+1-i}\left(\mc{F}^{a}_{\bullet+B}(\bb{W}_{(A,1^B)})\right).\]
The conclusion of Theorem~\ref{thm:Ext-hook=Ext-divided} now follows from the existence of the quasi-isomorphism \eqref{eq:def-PhiB-pre}.
\end{proof}

\section{Some explicit formulas for $\Ext$ groups}
\label{sec:explicit-Ext}

In this section $\kk$ denotes a field of characteristic $p>0$. Our goal is to provide explicit formulas between extension groups verifying Theorem~\ref{thm:dimExt-hook-2row} and Corollary~\ref{cor:Ext-vs-block}. To that end, we define the polynomials
\[ E_{m,n}(t) = \sum_{j\geq 0}\dim_{\kk}\Ext^j\left(\Sym^{m+n},\bb{S}_{(m,n)}\right) \cdot t^j,\]
and the power series
\[ \mc{E}_k(t,u) = \sum_{n\geq 0} E_{n+k,n}(t)\cdot u^n.\]
The first objective is to prove that $\mc{E}_k(t,u)$ can be rewritten as in \eqref{eq:E-pow-ser}. We recall from \cite{RV}*{(1.5)--(1.6)} the polynomials $H_{a,b}(t)$ and the power series $\mc{N}_b(t,u)$, which can be written as
\begin{equation}\label{eq:defHab-Nb}
 H_{a,b}(t) = \sum_{j\geq 0}\dim_{\kk}H^j_{st}\left(\bb{S}_{(a,1^b)}\right)\cdot t^j,\qquad \mc{N}_b(t,u) = \sum_{a\geq 1}\frac{H_{a,b}(t)}{t^b}\cdot u^{a+b}.
\end{equation}

\begin{lemma}\label{lem:Hab-vs-Emn}
 We have the identifications
 \[ E_{m,n}(t) = \frac{H_{n+1,m-n}(t)}{t^{m-n+1}} \qquad\text{and}\qquad \mc{E}_k(t,u) = \frac{\mc{N}_k(t,u)}{t\cdot u^{k+1}}.\]
\end{lemma}

\begin{proof} We have the series of isomorphisms
\[ 
\begin{aligned}
\Ext^j\left(\Sym^{m+n},\bb{S}_{(m,n)}\right) &\overset{\eqref{eq:ext-weyl-schur=eqs}}{\simeq}\Ext^j\left(\bb{S}_{(2^n,1^{m-n})},\bw^{m+n}\right) \overset{\eqref{eq:strange-duality}}{\simeq}
\Ext^{n-j}\left(\bb{S}_{(n+1,1^{m-n})},\bw^{m+1}\right) \\
&\overset{\text{Theorem~\ref{thm:stcoh-transpose-duality}}}{\simeq} H^{m+1-n+j}_{st}\left(\bb{S}_{(n+1,1^{m-n})}\right)\\
\end{aligned}
\]
from which the identity $H_{n+1,m-n}(t)=t^{m-n+1}\cdot E_{m,n}(t)$ follows. We then get $\mc{N}_k(t,u) = t\cdot u^{k+1}\cdot \mc{E}_k(t,u)$ by setting $m=n+k$, and substituting $a=n+1$ and $b=k$ in \eqref{eq:defHab-Nb}.
\end{proof}

It follows from Lemma~\ref{lem:Hab-vs-Emn} and \cite{RV}*{(1.8)--(1.9)} that $\mc{E}_k(t,u)$ satisfy the following recursive relations: if we write $k=l\cdot p + k_0$ with $l\geq 0$ and $0\leq k_0\leq p-1$, then
\begin{equation}\label{eq:rec-Ektu}
 \mc{E}_k(t,u) = \begin{cases}
 \mc{E}_l(t,u^p) & \text{if }k_0=p-1,\\
 u^{p-1-k_0} \cdot \mc{E}_l(t,u^p) + (1+t\cdot u^{p-1-k_0})\cdot \mc{A}(t,u) & \text{if }0\leq i\leq p-2.
 \end{cases}
\end{equation}
Notice that these recursions uniquely determine $\mc{E}_k(t,u)$: indeed, when $k=0$ we get
\[ \mc{E}_0(t,u) = u^{p-1}\cdot\mc{E}_0(t,u^p) + (1+t\cdot u^{p-1})\cdot\mc{A}(t,u),\]
which determines $\mc{E}_0(t,u)$ uniquely, which in turn determines by induction every $\mc{E}_k(t,u)$. To show that $\mc{E}_k(t,u)$ is given by the formula \eqref{eq:E-pow-ser}, we provisionally set
\begin{equation}\label{eq:def-E'ktu}
\mc{E}'_k(t,u) = \sum_{k_i\neq p-1} \left(u^{\ol{k}(i-1)}+t\cdot u^{\ol{k}(i)}\right) \cdot \mc{A}(t,u^{p^i})
\end{equation}
and show that $\mc{E}'_k(t,u)$ also satisfies the recursive relations \eqref{eq:rec-Ektu}. Notice that if $k=l\cdot p + k_0$ then
\[ k_{i+1}=l_i\quad\text{ and }\quad\ol{k}(i+1) = (p-1-k_0) + p\cdot\ol{l}(i),\quad\text{ for }i\geq 0.\]
It follows that
\begin{equation}\label{eq:truncate-E'ktu}
\begin{aligned}
 \sum_{\substack{i\geq 0 \\ k_{i+1}\neq p-1}} \left(u^{\ol{k}(i)}+t\cdot u^{\ol{k}(i+1)}\right) \cdot \mc{A}(t,u^{p^{i+1}}) &= u^{p-1-k_0}\cdot \sum_{\substack{i\geq 0 \\ l_i\neq p-1}} \left(u^{p\cdot\ol{l}(i-1)}+t\cdot u^{p\cdot\ol{l}(i)}\right) \cdot \mc{A}(t,u^{p\cdot p^i}) \\
 &= u^{p-1-k_0} \cdot \mc{E}'_l(t,u^p).\\
 \end{aligned}
\end{equation}
If $k_0=p-1$ then the left side of \eqref{eq:truncate-E'ktu} agrees with the right side of \eqref{eq:def-E'ktu}, hence $\mc{E}'_k(t,u) = \mc{E}'_l(t,u^p)$. If $k_0\neq p-1$ then the left side of the displayed equation differs from the right side of \eqref{eq:def-E'ktu} by 
\[\left(u^{\ol{k}(-1)}+t\cdot u^{\ol{k}(0)}\right) \cdot \mc{A}(t,u^{p^0}) = (1+t\cdot u^{p-1-k_0})\cdot \mc{A}(t,u),\]
so the recursive relations \eqref{eq:rec-Ektu} hold for $\mc{E}'_k(t,u)$. This means that $\mc{E}_k(t,u)=\mc{E}'_k(t,u)$, as desired.

\begin{proof}[Proof of Theorem~\ref{thm:dimExt-hook-2row}] Since the $\Ext$ groups vanish for $A<a$, we may assume that $A\geq a$, hence $b\geq B$. 

For part (1), we can further assume using invariance under column removal that $B=0$ and therefore $b=A-a$. This shows that $e^j(\ll,\mu)$ equals the coefficient of $t^j$ in $E_{a,b}(t)$. Writing $k=a-b$, this also equals the coefficient of $t^j\cdot u^b=t^j\cdot u^{A-a}$ in $\mc{E}_k(t,u)=\mc{E}_{a-b}(t,u)$, as desired. Part (2) is equivalent to (1) via \eqref{eq:ext-weyl-schur=eqs}.

For part (3) we use Theorem~\ref{thm:introExtComputation} to conclude that $e^j(\ll,\mu)$ equals the coefficient of $t^{d-j}$ in $H_{\Delta+1,d-\Delta-1}(t)$, where $\Delta=A-a$ and $d=a+b=A+B$. By Lemma~\ref{lem:Hab-vs-Emn} with $m=d-1$ and $n=\Delta$, this equals the coefficient of $t^{\Delta-j}$ in $E_{d-1,\Delta}(t)$, which in turn equals the coefficient of $t^{A-a-j}\cdot u^{A-a}$ in $\mc{E}_{a+B-1}(t,u)$, as desired.
\end{proof}

\begin{proof}[Proof of Corollary~\ref{cor:Ext-vs-block}] General block theory implies that if $\ll,\mu$ are in different blocks then $\Ext^{j}(\bb{S}_{\ll},\bb{S}_{\mu})=0$ for all $j$, so we only need to address the converse.

If we let $m=a-B$ and $n=b-B$, then the existence of an index $j$ with $\Ext^{j}(\bb{S}_{\ll},\bb{S}_{\mu})\neq 0$ is equivalent to the condition that $E_{m,n}(t)\neq 0$. By Lemma~\ref{lem:Hab-vs-Emn}, this is further equivalent to the condition that $H_{n+1,m-n}(t)\neq 0$. We write
\[ m-n = -1 + c\cdot q,\quad\text{ with }q=p^r,\ p\nmid c,\]
and conclude using \cite{RV}*{(1.10)} that 
\begin{equation}\label{eq:Hab-nonzero}
 H_{n+1,m-n}(t)\neq 0 \quad\Longleftrightarrow\quad pq | n \quad\text{or}\quad pq | m+1.
\end{equation}

Notice that $q$ is the largest power of $p$ which divides $m-n+1=a-b+1$. If $\ll=(A,B)$ and $\mu=(a,b)$ are in the same block, then \cite{donkin-blocks} implies that $q$ is also the largest power of $p$ which divides $A-B+1$. In addition, one of the following conditions must hold:
\begin{enumerate}
 \item $\ll_i\equiv\mu_i\ (\text{mod }pq)$ for $i=1,2$, or
 \item $\ll_1-1\equiv \mu_2-2\ (\text{mod }pq)$ and $\ll_2-2\equiv \mu_1-1\ (\text{mod }pq)$.
\end{enumerate}
In case (1) we get $pq|A-a=n$, and in case (2) we get $pq|a-B+1 = m+1$, so the desired non-vanishing follows from \eqref{eq:Hab-nonzero}. 
\end{proof}

\section*{Acknowledgements}
Experiments with Macaulay2 \cite{GS} and GAP \cite{doty-gap} have provided many valuable insights. The authors would like to thank Mark Behrens, Ben Briggs, Stephen Donkin, Karthik Ganapathy, Srikanth Iyengar, Henning Krause and Antoine Touz\'e for helpful communications regarding various aspects of this project.  Raicu acknowledges the support of the National Science Foundation Grant DMS-2302341. VandeBogert acknowledges the support of the National Science Foundation Grant DMS-2202871.

\bibliography{biblio}

\end{document}